\documentclass[review]{elsarticle}
\makeatletter
\def\ps@pprintTitle{%
 \let\@oddhead\@empty
 \let\@evenhead\@empty
 \def\@oddfoot{}%
 \let\@evenfoot\@oddfoot}
\usepackage{array}
\newcolumntype{P}[1]{>{\centering\arraybackslash}p{#1}}

\usepackage{tabularx}
\usepackage{amsmath,amsfonts,bm}

\def\eqref#1{equation~\ref{#1}}
\def\1{\bm{1}}

\DeclareMathAlphabet{\mathsfit}{\encodingdefault}{\sfdefault}{m}{sl}
\SetMathAlphabet{\mathsfit}{bold}{\encodingdefault}{\sfdefault}{bx}{n}

\newcommand{\E}{\mathbb{E}}

\newcommand{\R}{\mathbb{R}}

\newcommand{\KL}{D_{\mathrm{KL}}}

\usepackage{subcaption}
\usepackage{cuted}
\usepackage{caption}
\usepackage{amssymb}
\usepackage{enumitem}  
\usepackage{varwidth}
\usepackage{tasks}
\usepackage{wrapfig}
\usepackage{graphicx}
\usepackage{subcaption}
\usepackage[utf8]{inputenc}
\usepackage[english]{babel}
\usepackage{amsthm}
\usepackage{hyperref}
\usepackage{url}
\usepackage{float}
\usepackage{booktabs}       %professional-quality tables
\usepackage{multirow}

\newcommand{\beps}{{\boldsymbol{\epsilon}}}
\newcommand{\X}{{\textbf{X}}}

\newcommand{\W}{\mathbf{W}}
\newcommand{\e}{\mathbf{e}}

\newcommand{\di}{{\mathbf{D}}}
\newcommand{\sdi}{{\mathbf{d}}}
\newcommand{\II}{{\mathbf{I}}}

\newcommand{\g}{\textbf{G}}

\newcommand{\I}{\textbf{I}}
\newcommand{\eg}{\widehat{\g}}
\newcommand{\rg}{\textbf{G}}
\newcommand{\seg}{\widehat{\textbf{g}}}
\newcommand{\srg}{\textbf{g}}
\newcommand{\bg}{\underline{\textbf{G}}}
\newcommand{\sbg}{\underline{\textbf{g}}}

\newcommand{\bPsi}{\boldsymbol{\Psi}}
\newcommand{\Vt}{\boldsymbol{\Lambda}^{-\frac{1}{2}}}

\newcommand{\sVtori}{{\tilde{\textbf{v}}}}
\newcommand{\M}{{\textbf{M}}}

\newcommand{\D}{\mathcal{D}}

\newcommand{\z}{\textbf{z}}
\newcommand{\y}{\textbf{y}}
\newcommand{\Z}{\textbf{Z}}

\newcommand{\wi}{\textbf{v}}
\newcommand{\m}{\textbf{m}}
\newcommand{\bxi}{\boldsymbol{\xi}}
\newcommand{\btheta}{{\boldsymbol{\theta}}}
\newcommand{\balpha}{{\boldsymbol{\alpha}}}

\newcommand{\x}{{\mathbf{x}}}

\newcommand{\lpower}{\boldsymbol{\Lambda}}

\newcommand{\ADAM}{{\textsc{Adam}$^3$~}}
\newcommand{\DADAM}{{\textsc{Dadam}$^3$~}}

\newcommand{\CADAM}{{\textsc{Cadam}$^3$~}}
\newcommand{\DSGD}{{\textsc{Dosg}~}}
\newcommand{\CSGD}{{\textsc{Cosg}~}}
\usepackage{tikz-cd}
\newtheorem{theorem}{Theorem}[section]
\newtheorem{corollary}[theorem]{Corollary}
\newtheorem{lemma}[theorem]{Lemma}

\newtheorem{assum}{Assumption}
\newtheorem{deff}[theorem]{Definition}
\newtheorem{rmk}[theorem]{Remark}
\usepackage[ruled,vlined]{algorithm2e}

\begin{document}

\begin{frontmatter}

\title{ A Decentralized Adaptive Momentum Method for Solving a Class of Min-Max Optimization Problems}

\author[mymainaddress]{Babak Barazandeh\corref{mycorrespondingauthor}}
\cortext[mycorrespondingauthor]{Corresponding author}
\ead{bbarazandeh@splunk.com}

\author[mysecondaryaddress]{Tianjian Huang}

\author[mythirdddress]{George Michailidis}

\address[mymainaddress]{Splunk}

\address[mysecondaryaddress]{University of Southern California}
\address[mythirdddress]{University of Florida}

\begin{abstract}

Min-max saddle point games have recently been intensely studied, due to their wide range of applications, including training Generative Adversarial Networks~(GANs). However, most of the recent efforts for solving them are limited to special regimes such as convex-concave games. Further, it is customarily assumed that the underlying optimization problem is solved either by a single machine or in the case of multiple machines connected in centralized fashion, wherein each one communicates with a central node. The latter approach becomes challenging, when the underlying communications network has low bandwidth. In addition, privacy considerations may dictate that certain nodes can communicate with a subset of other nodes. Hence, it is of interest to develop methods that solve min-max games in a decentralized manner. To that end, we develop a decentralized adaptive momentum (ADAM)-type algorithm for solving min-max optimization problem under the condition that the objective function satisfies a Minty Variational Inequality condition, which is a generalization to convex-concave case. The proposed method overcomes shortcomings of recent non-adaptive gradient-based decentralized algorithms for min-max optimization problems that do not perform well in practice and require careful tuning. In this paper, we obtain non-asymptotic rates of convergence of the proposed algorithm (coined \DADAM) for finding a (stochastic)~first-order Nash equilibrium point and subsequently evaluate its performance on training GANs. The extensive empirical evaluation shows that \DADAM outperforms recently developed methods, including decentralized optimistic stochastic gradient for solving such min-max problems.
\end{abstract}

\begin{keyword}
first-order Nash equilibria, stationary points, distributed optimization, non-convex min-max problems. variational inequality
\end{keyword}

\end{frontmatter}

\section{Introduction}
Recent growth in the size and complexity of data and the emergence of new machine learning problems have resulted in the need to solve large scale optimization problems. From early on, stochastic gradient descent (SGD) and its variants have been the major tool for solving such large-scale problems~\cite{robbins1951stochastic,polyak1992acceleration,nemirovski2009robust,moulines2011non}. However, in order to ensure convergence, SGD requires a decaying step-size to zero which results in slow convergence in practice~\cite{johnson2013accelerating}. To overcome this issue, different gradient based adaptive methods that automatically adjust the step-size in each iteration, such as AdaGrad~\cite{duchi2011adaptive}, Adadelta~\cite{zeiler2012adadelta} and RMSprop~\cite{hinton2012lecture} have been proposed in the literature. Recently, \cite{kingma2014adam} combined the adaptive idea with the momentum method and proposed the widely used Adam algorithm. 
Despite wide applicability of adaptive and momentum based methods, it is worth noting that they are mainly designed for solving classical convex~\cite{levy2017online,kavis2019unixgrad,ene2020adaptive} or non-convex optimization problems~\cite{zou2018weighted,ward2019adagrad,zou2019sufficient}. Further, these methods were originally operating in a centralized fashion; however, recent developments ~\cite{nazari2019dadam,shi2015extra,nedic2017achieving,tang2018d,jiang2017collaborative} have adapted them to decentralized problems.

New machine learning problems, such as training GANs, require solving a new class of min-max saddle point games~\cite{razaviyayn2020nonconvex, barazandeh2019training}. They correspond to a zero-sum two player game, wherein one player's goal is to increase the objective function, while the other's to decrease it. A solution to this problem corresponds to finding a Nash equilibrium point~\cite{nash1950equilibrium}, for which several algorithms including Mirror-Prox~\cite{nemirovski2004prox} have been proposed, provided that the nature of the objective function is convex-concave. In general, solving min-max optimization problem is difficult which keeps other approaches attractive~\cite{barazandeh2021efficient, barazandeh2018behavior, barazandeh2017robust, bastani2018fault}. This is due to the fact that finding a Nash equilibrium point is computationally NP-hard in general~\cite{nouiehed2019solving}; therefore, the focus has shifted to finding a first-order Nash point (see upcoming section for more details) and numerous algorithms have been proposed to find such a point in non-convex-concave games~\cite{ostrovskii2020efficient}, and non-convex-non-concave games under Polyak-\L{}ojasiewicz~(PL)~\cite{yang2020global} and Minty Variational Inequality~(VI) ~\cite{liu2019towards} conditions. However, most of the proposed algorithms are centralized in nature. Thus, the goal of this work is to design a \textit{decentralized adaptive momentum algorithm} for solving a class of non-convex-non-concave min-max games satisfying the Minty VI condition. Due to the adaptive structure of the proposed algorithm, it automatically updates the step-size at each iteration and hence it does not require a decaying step-size.      
The remainder of the paper is organized as follows. Section~\ref{sec:related_work} provides a brief overview of related work, Section~\ref{sec:Problem_for} presents the structure of the decentralized min-max problem, while Section~\ref{sec:algorithm} introduces the proposed algorithm.  Section~\ref{sec:CA} establishes a non-asymptotic convergence rate for the algorithm and finally, Section~\ref{sec:NS} evaluates its empirical performance.

\section{Related Work}\label{sec:related_work}

\subsection{Decentralized Versus Centralized Minimization Problems} In centralized optimization problems, there exists a central node which holds information on the model and updates the model's parameters (for example, the weights of a neural network) by appropriately aggregating information (e.g., local gradients calculated by each node) received from all other nodes~\cite{suresh2017distributed,yuan2016convergence}). A centralized setting is not desirable whenever there are network bandwidth limitations, or privacy considerations~\cite{lian2017can, liu2019decentralized, tsianos2012consensus}. On the other hand, in a decentralized setting~\cite{lian2017can, tsitsiklis1986distributed,nedic2009distributed,shi2015extra}, every node possesses its own local copy of the model and nodes are allowed to communicate with any subset of nodes they desire. This setup addresses both network bandwidth issues and data privacy concerns~\cite{ram2009asynchronous,chen2012diffusion,ling2012decentralized,olfati2007consensus}. Further, decentralized approaches are used in a wide range of other applications, including smart grid management problems~\cite{kekatos2012distributed,dall2013distributed}, wireless sensor networks~\cite{predd2006distributed,zhao2002information} and dictionary learning~\cite{wai2015consensus,zhao2016distributed,daneshmand2019decentralized}.

\subsection{Min-Max Saddle Point Problems}
For such problems with a convex-concave objective function, numerous algorithms including Primal-Dual based methods~\cite{hamedani2018primal,hamedani2018iteration}, Frank-Wolfe type ~\cite{gidel2017frank,abernethy2017frank} and optimistic mirror descent algorithms~\cite{mertikopoulos2018optimistic,mokhtari2020unified} have been proposed to find a Nash equilibrium point. As previously mentioned, finding such a point is NP-hard in general and as a result recent work has focused on first-order Nash equilibrium points~\cite{nouiehed2019solving}. For example,~\cite{rafique2018non} studies the problem for a weakly-convex-concave regime, while \cite{ostrovskii2020efficient, nouiehed2019solving,lu2020hybrid} examine a broader class of non-convex-concave regimes.~\cite{nouiehed2019solving} and~\cite{yang2020global} represent early work aimed to solve the saddle point game for a non-convex-non-concave objective function by assuming that the latter satisfies a Polyak-\L{}ojasiewicz~(PL) condition in the argument of one of the players. In this work, we study another class of non-convex-non-concave games, wherein the objective function satisfies the Minty Variational Inequality~(VI) condition~\cite{cottle2009linear}.

The strong recent interest in solving min-max saddle point games is due to emerging machine learning applications. One of them is training GANs ~\cite{goodfellow2014generative, arjovsky2017wasserstein} that have exhibited very good performance in high resolution image and creative language generation~\cite{wang2018high} and also common sense reasoning~\cite{saeed2019creative}. 
GANs aim to train a generative model that outputs samples that are as close as possible to real samples in some metric that captures proximity of the respective distributions. To that end, GANs convert learning a generative model to a zero-sum game comprising of two players: the generator and the discriminator. The generator aims to generate real-like samples from random (usually Gaussian) noise input data that resemble closely the real samples \cite{heusel2017gans}, while the discriminator aims to distinguish the artificial data generated by the generator from the real samples. This game is commonly formulated through a min-max optimization problem and in practice both the generator and the discriminator are modeled using Deep Neural Networks (DNN)~\cite{wang2017generative}.
%and the objective is to find a Nash equilibrium point of this zero-sum game such that generator in this point can output data samples that resembles the real data~\cite{heusel2017gans}. 
%Specifically, training of GANs requires solving a min-max saddle point game where loss~(objective) function is non-convex-non-concave. 
Further, different choices for the loss function have been used, resulting in generic~\cite{goodfellow2014generative}, Wasserstein~\cite{arjovsky2017wasserstein}, second-order Wasserstein~\cite{feizi2017understanding}, least squares GANs~\cite{mao2017least} and $f$-GANs~\cite{nowozin2016f}. 

Besides GANs, many other problems in machine learning and signal processing require solving a min-max saddle point game. As an example, consider recent efforts focusing on fair machine learning~\cite{madras2018learning}. Recent studies show that machine learning algorithms can result in systematic discrimination of people belonging to different minority groups~\cite{xu2018fairgan,madras2018learning}. A number of proposed methods in the literature aim to address this issue based on an adversarial framework, that requires solving a min-max saddle point game~\cite{madras2018learning}. Similar approaches are used for resource allocation problems in wireless commutation systems, wherein the goal is to maximize the minimum transmission rate among all users, which also requires solving min-max games~\cite{liu2013max}.

\subsection{Decentralized Min-Max Problems} 

Such problems have been studied when the objective function is assumed to be convex-concave and multiple algorithms such as sub-gradient based methods~\cite{mateos2015distributed} and primal-dual methods~\cite{wai2018multi} have been proposed. ~\cite{tsaknakis2020decentralized} studies the problem when the objective function is non-convex-(strongly) concave and proposes a gradient-based method to find the first-order Nash point. There have been few efforts to study this problem for a broader class of non-convex-non-concave games under the Minty VI condition, which resulted in algorithms with theoretical guarantees for convergence to first-order Nash equilibrium points. As an example,~\cite{proximal} proposed a proximal point-type method for solving such a game. However, that paper assumes that the resulting sub-problem has a closed-form solution which does not often hold in practice. To overcome this issue,~\cite{liu2019decentralized} proposed a non-adaptive gradient-based method for finding a first-order Nash point. However, non-adaptive methods have been very slow and generate low performance solutions in solving optimization problems that include deep neural networks~\cite{duchi2011adaptive,mcmahan2010adaptive}. This might be the reason~\cite{liu2019decentralized} did not use their proposed non-adaptive method in the numerical experiments involving training of GANs. In this work, we propose an adaptive momentum method for solving a broad class of non-convex-non-concave decentralized min-max saddle point games wherein the objective function satisfies the Minty VI condition.

\section{Preliminaries and Problem Formulation}\label{sec:Problem_for}
A general min-max saddle point game is defined as
\begin{equation}\label{eq:main_game}
\min\limits_{\btheta }\max\limits_{\balpha} F(\btheta, \balpha),
\end{equation} 

% \begin{equation}\label{eq:main_game}
% \min\limits_{\btheta }\max\limits_{\balpha} F(\btheta, \balpha) = \E_{\bxi \sim \mathcal{D}}[f(\btheta,\balpha;\bxi)],
% \end{equation} 

where $\btheta \in \mathbb{R}^{p_1}, \balpha \in \mathbb{R}^{p_2}$ and $F(\btheta,\balpha)$ is non-convex-non-concave, i.e., it is non-convex in $\btheta$ for any given $\balpha$ and is non-concave in $\balpha$ for any given $\btheta$.  This problem can be considered as a two player game in which one player's goal is to increase the value of the objective function, while the other player aims to decrease it. The objective of this game is to find a point, such that no player can improve her own payoff by solely changing her own strategy; such a point corresponds to a Nash equilibrium~\citep{nash1950equilibrium} formally defined next.

\begin{deff}(NE) A point $(\btheta^*, \balpha^*) \in \mathbb{R}^{p_1} \times \mathbb{R}^{p_2}$ is a Nash equilibrium~(NE) of Game~\ref{eq:main_game} if 
\begin{equation*}
F(\btheta^*, \balpha) \leq  F(\btheta^*,\balpha^*) \leq F(\btheta, \balpha^*), \; \forall (\btheta, \balpha) \in \mathbb{R}^{p_1} \times \mathbb{R}^{p_2}.
\end{equation*}
\end{deff}

The above definition implies that $\btheta^*$ is a global minimum of $F(\cdot, \balpha^*)$ and $\balpha^*$ is a global maximum of $F(\btheta^*, \cdot)$. In the convex-concave regime where $F(\btheta,\balpha)$ is convex in $\btheta$ for any given $\balpha$ and concave in $\balpha$ for any give $\btheta$, a NE always exists~\citep{jin2019minmax} and there are several algorithms for finding such a point~\citep{gidel2017frank,hamedani2018iteration}. However, finding a Nash equilibrium is difficult in general non-convex-non-concave setting~\citep{daskalakis2017training,jin2019minmax} and it may not even exist~\citep{farnia2020gans}. As a result, since we are considering the general non-convex-non-concave  regime, we focus on finding a \textit{first-order Nash equilibrium} (FNE) point~\citep{nouiehed2019solving,barazandeh2020solving} defined next.

\begin{deff}\label{deff:FN}
(FNE)~A point $(\btheta^*, \balpha^*) \in \mathbb{R}^{p_1} \times \mathbb{R}^{p_2}$ is a first-order Nash equilibrium point of the Game in equation~\ref{eq:main_game}, if $ \nabla_{\btheta}F(\btheta,\balpha) = 0$ and $\nabla_{\balpha}F(\btheta,\balpha) = 0$. 
\end{deff}
%\begin{deff}
%Let $\g(\btheta,\balpha) = [\nabla_{\btheta}F(\btheta,\balpha), -\nabla_{\balpha}F(\btheta,\balpha)]$. A point $ (\btheta^*, \balpha^*) \in \mathbb{R}^{p_1} \times \mathbb{R}^{p_2}$ is a first-order Nash equilibrium (FNE) point of Game~\ref{eq:main_game} if $\g(\btheta^*, \balpha^*) = 0$. 
%\end{deff}
Based on the above definition, also used in~\citep{pang2016unified,pang2011nonconvex}, at the FNE point, each player satisfies the first-order optimality condition of its own objective function when the strategy of the other player is fixed. 

The key objective of this work is to solve a stochastic and decentralized version of the problem in~\eqref{eq:main_game}, defined as
\begin{equation}\label{eq:main_game_s}
\min\limits_{\btheta }\max\limits_{\balpha} F(\btheta, \balpha) = \frac{1}{M} \sum_{i = 1}^M F_i(\btheta, \balpha) 
\end{equation} 
wherein $F_i(\btheta, \balpha) = \E_{\bxi \sim \D_i}[f(\btheta, \balpha;\bxi)]$ is the loss function, $\bxi$ is a random variable corresponding to the data samples drawn to calculate the gradients of the function $f$ and $\D_i$ is a predefined distribution for node $i \in \{1,\cdots, M\}$. As mentioned in~\cite{lian2017can}, there are two different strategies to achieve~\eqref{eq:main_game_s}; assuming that all distributions $\D_i$'s are the same as $\D$ or each node has its own sampling strategy. In this work, similar to~\cite{liu2019decentralized}, we will work with the first strategy and assume $F_i(\btheta, \balpha) = \E_{\bxi \sim \D}[f(\btheta, \balpha;\bxi)], \forall \ i $. In this setting, each node has its own (local) copy of the model and updates it internally which preserves its privacy. However, nodes can connect to their selected neighbours and exchange information with them using a mixing matrix $\W$. Specifically, consider $M$ nodes that are connected to each other by a graph $\mathcal{G} = (\mathcal{V}, \mathcal{E})$, where $\mathcal{V} = \{1,\cdots, M\}$ and $\mathcal{E} \subset \mathcal{V} \times \mathcal{V}$ correspond to the vertex and edge sets of this graph, respectively. The element in the $i^{th}$ row and $j^{th}$ column of matrix $\W$, denoted by $[\W]_{ij}$, indicate how nodes $i$ and $j$ communicate. If $[\W]_{ij}=0$, then the corresponding nodes are disconnected. In the decentralized set-up, node $i$ in graph $\mathcal{G}$ is only allowed to communicate with those in its neighborhood, $\mathcal{N}_i = \{j | [\W]_{ij} > 0\}.$

In practice, we use iterative algorithms for finding a FNE of stochastic games. As a result, we evaluate the performance of different stochastic iterative algorithms using the following approximate-FNE.

\begin{deff}[$\epsilon$-Stochastic First-Order Nash Equilibrium (SFNE)] \label{deff:EFN} A random variable $(\btheta^*, \balpha^*)$ is an approximate SFNE~($\epsilon$-SFNE) point of the game defined in~\eqref{eq:main_game_s}, if 
$
\E\left[\|\nabla F(\btheta^*, \balpha^*)\|^2 \right] \leq \epsilon^2,
$
where the expectation is taken over the distribution of the random variable $(\btheta^*, \balpha^*)$. 
\end{deff}

The randomness of the desired points in Definition~\ref{deff:EFN} follows from the fact that in practice, we use iterative algorithms to find them and these algorithms have only access to stochastic gradients of the objective function.
In this work, the goal is to find an $\epsilon-$SFNE point for Game~\ref{eq:main_game_s} using an iterative method based on adaptive momentum.

\textbf{Notation:} Throughout the paper, we define $\y = (\btheta,\balpha) \in \mathbb{R}^{p_1} \times \mathbb{R}^{p_2}$, $d = p_1 + p_2$ and denote the objective functions by $F(\y)$ and $f(\y;\bxi)$. Further, we introduce $ \nabla F(\y) \equiv [\nabla_{\btheta}F(\btheta,\balpha), -\nabla_{\balpha}F(\btheta,\balpha)]$ and $\nabla f(\y;\bxi) \equiv [\nabla_{\btheta}f(\btheta,\balpha;\bxi), -\nabla_{\balpha}f(\btheta,\balpha;\xi)] $ to denote their corresponding gradients.

\section{A Decentralized Adaptive Momentum Algorithm}\label{sec:algorithm}
We start by introducing the proposed algorithm, coined ADAptive Momentum Min-Max~({\ADAM}) on a single computing node, in order to gain insight into its workings. Subsequently, we present its Decentralized variant.

\textbf{Single Computing Node:} As detailed in Algorithm~\ref{alg:0}, at iteration $k$, \ADAM generates two sequences $\z_k$ and $\x_k$, where $\x_k$ is an ancillary variable and the (stochastic) gradient is calculated at $\z_k$. After calculating $\z_k$, a mini-batch of data of size $m$ is obtained -$\bxi_{k} = (\bxi_{k}^{1}, \cdots, \bxi_{k}^{m})$- that is used to calculate the stochastic gradient of the objective function at point $\z_k$, i.e., $\seg_{k} =\frac{1}{m} \sum_{j = 1}^m \nabla f(\z_{k};\bxi_{k}^{j})$. Subsequently, $\m_k$ and $\sVtori_k$, the exponential moving averages of the gradient and the squared gradient, respectively, are computed. Finally, the variable $\x_k$ is updated by the scaled momentum ${\sVtori_{k}^{-\frac{1}{2}}} \odot \m_{k}$. This method can be considered as a combination of the extra-gradient method~\citep{iusem2017extragradient} and AMSGrad~\citep{reddi2019convergence} used for minimization problems. The focus of this paper is on the decentralized problem (Algorithm~\ref{alg:1}) and we only mention the most related corollary for Algorithm~\ref{alg:0} in the sequel. Theoretical results and performance of \ADAM on a single computing node are given in~\cite{barazandeh2021solving}. 
\begin{corollary}(Rephrased from~\cite{barazandeh2021solving})\label{cor:algorhtim1}
Algorithm~\ref{alg:0} requires a total of $\mathcal{O} (\epsilon^{-4})$ gradient evaluations of the objective function to find an $\epsilon$-SFNE point  of Game~\ref{eq:main_game}.
\end{corollary}
\begin{center}
\begin{algorithm}[H]
    \SetAlgoLined
 	\SetKwInOut{Input}{Input}
	\SetKwInOut{Output}{Output}
\Input{$\{\beta_{1,k}\}_{k=1}^N, \beta_2, \beta_3 \in [0,1)$, $m \in \mathbb{N}$, and $ \eta \in \R_+$\;}
\BlankLine
Initialize $\z_{0} = \x_{0}  = \m_{0}= \wi_{0} = \textbf{d}_{0}= \bm{0}_{d}$. 
\BlankLine
		\For{$ k =1:N$}
	{
	 
	    $\z_{k} =  \x_{k-1}  - \eta \textbf{d}_{k-1}$; %\qquad{{\small\texttt{//update local estimate $\z_{i,k}$}}}\\
	    
	    Draw $\bxi_{k} = (\bxi_{k}^{1}, \cdots, \bxi_{k}^{m})$ from $\mathcal{D}$, and set $\seg_{k} =\frac{1}{m} \sum_{j = 1}^m \nabla f(\z_{k};\bxi_{k}^{j})$ %\qquad{\small\texttt{//estimate the stochastic gradient\\
	    
	    $\m_{k} = \beta_{1,k} \m_{k-1} + (1-\beta_{1,k}) \seg_{k}$ %\\ \qquad{\small\texttt{//update biased 1st moment estimate}}\\
	    
	    $ \wi_{k} =   \beta_2 \wi_{k-1} + (1-\beta_2) \seg_{k} \odot \seg_{k}$% \qquad{\small\texttt{//update biased 2nd moment estimate}}\\
	    
	    $ \sVtori_{k} = \beta_3\sVtori_{k-1} + (1-\beta_{3})\max^*(\sVtori_{k-1}, \wi_{k} )$ % \qquad{\small\texttt{//normalize  2nd moment }}\\
	    
	    $\textbf{d}_{k} =  {\sVtori_{k}^{-\frac{1}{2}}} \odot \m_{k}$ %\qquad{\small\texttt{//scale the adaptive gradient}}\\
	    
	    $\x_{k} =  \x_{k-1} - \eta \;\textbf{d}_{k}$
	    %\qquad{\small\texttt{//update local estimate $\x_{i,k}$}}
	}
     \caption{ADAptive Momentum Min-Max~({\ADAM})}
     \label{alg:0}
     $\odot$: Element-wise vector multiplication, $^*$ : \text{Element-wise max operator}
    \end{algorithm}
\end{center}

\textbf{Multiple Computing Nodes:} 
%Consider $M$ nodes that are connected through a graph $\mathcal{G} = (\mathcal{V}, \mathcal{E})$, where $\mathcal{V} = \{1,\cdots, M\}$ and $\mathcal{E} \subset \mathcal{V} \times \mathcal{V}$ correspond to the vertex and edge sets of $\mathcal{G}$, respectively. The links between nodes are represented in compact form by matrix $\W$. Each element of this matrix, $[\W]_{ij}$, shows how nodes $i$ and $j$ communicate. If $[\W]_{ij}=0$, then the corresponding nodes are disconnected. In the decentralized set-up, node $i$ in graph $\mathcal{G}$ communicates only with nodes in its neighborhood $\mathcal{N}_i = \{j \ | \ [\W]_{ij} > 0\}.$ 
%
In this setting, each node $i$ at iteration $k$, has a local copy of all key variables, denoted  by
$\x_{i,k}, \z_{i,k}, \m_{i,k}$, and $\sVtori_{i,k}$, and samples a mini-batch of data of size $m$, $\bxi_{k,i} = (\bxi_{k,i}^{1}, \cdots, \bxi_{k,i}^{m})$. It uses the mini-batch to calculate the (stochastic) gradient of the objective function, i.e., $\seg_{i,k} = \frac{1}{m} \sum_{j = 1}^m f(\z_{i,k};\bxi_{k,i})$.   

To simplify the presentation, we use the following notation
\begin{equation}\label{eq:nota}
\begin{aligned}[c]
\X_k &= [\x_{1,k},\cdots, \x_{M,k}], \\
%\bar{\z}_t &= \frac{1}{M}\sum\limits_{i = 1}^M \z_{i,t} \in \mathbb{R}^{d \times M} \\
\Z_k &= [\z_{1,k},\cdots, \z_{M,k}],\\
\textbf{V}_k &= \left[\wi_{1,k},\cdots, \wi_{M,k}\right]
,\\
\boldsymbol{\Lambda}_k &= \left[\sVtori_{1,k},\cdots, \sVtori_{M,k}\right]
,
\end{aligned}
\qquad\qquad
\begin{aligned}[c]
\M_k &= \left[\m_{1,k},\cdots, \m_{M,k}\right],
\\
\di_k &= \left[\sdi_{1,k}\cdots, \sdi_{M,k}\right],
\\
\eg_k &=\left[\seg_{1,k};,\cdots, \seg_{M,k}\right],
\end{aligned}
\end{equation}
which corresponds to concatenations of the variables from all nodes at iteration $k$.
 
Algorithm~\ref{alg:1} depicts the steps of the proposed Decentralized ADAptive Momentum Min-Max~(\DADAM) algorithm.

% \begin{wrapfigure}{R}{0.5\textwidth}
\begin{center}
% \begin{small}
% \begin{minipage}{.52\textwidth}
    \begin{algorithm}[H]\label{alg:1}
        \SetAlgoLined
     	\SetKwInOut{Input}{Input}
    	\SetKwInOut{Output}{Output}
    \Input{$\{\beta_{1,k}\}_{k=1}^N, \beta_2, \beta_3 \in [0,1)$, neighbor list $\mathcal{N}_i$, mixing matrix $\W$ and a step-size $ \eta \in \R_+$}
    \BlankLine
    Initialize $\Z_0 = \X_0  = \M_0 =  \textbf{V}_0 =  \di_0 = \bm{0}_{d}$. 
    \BlankLine
    		\For{$ k =1:N$}
    	{
    	 
    	    $\Z_k =  (\X_{k-1}- \eta \di_{k-1})\W^t$ 
    	    
    	    $\M_k = \beta_{1,k} \M_{k-1} + (1-\beta_{1,k})\eg_{k}$
    	    
    	    $\textbf{V}_k = \beta_{2} \textbf{V}_{k-1} + (1-\beta_2) \eg_{k} \odot \eg_{k}$
    	    
    	    $\boldsymbol{\Lambda}_k = \beta_3 \boldsymbol{\Lambda}_{k-1} + (1-\beta_3) \max^{*} (\boldsymbol{\Lambda}_{k-1}, \textbf{V}_k)$
    	    
    	    %$\Vt_k = \boldsymbol{\Lambda}^{-\frac{1}{2}}$;
    	    
    	    $\di_k ={}^{*}\Vt_k \circ \M_k$
    	    
    	    $\X_k = (\X_{k-1} - \eta \di_{k})\W^t$
        }
         \caption{Decentralized ADAptive Momentum Min-Max~({\DADAM})}
         $\odot$: Hadamard product, ${}^{*}$: power and max operators are element-wise
        \end{algorithm}
% \end{minipage}
% \end{small}
\end{center}
% \end{wrapfigure}
The following remarks about \DADAM are in order:
\\
\textbf{(1)}
The power and the maximum operators are applied element-wise. We may also add a small positive constant $\epsilon$ to each element of $\textbf{V}_k$ to prevent division by zero~\citep{kingma2014adam}. Further, each node samples a mini-batch of size $m$ to estimate the gradient's value and generate the final $\eg_k$. \\
\textbf{(2)}
\DADAM computes adaptive learning rates from estimates of the second moments of the gradients, similar to the approach developed by \citep{nazari2019dadam} for minimization problems. In particular, compared to AMSGrad, it uses a larger learning rate and decays the effects of previous gradient over-time. The decay parameter $\beta_{3}$ in the algorithm enables us to establish convergence result similar to AMSGrad ($\beta_3= 0$) and maintain the efficiency of Adam at the same time.
\\
\textbf{(3)}
The (local) averaging step in the algorithm is performed $t$ times at each iteration which is obtained by $\W^t$. The logarithmic magnitude of $t$ is only needed to provide theoretical convergence guarantees, while in practice a single step of averaging is adequate. 
% \\
% \textbf{(4)}  By choosing the mixing matrix to be fully connected matrix,i.e., $\frac{1}{M^2}\W = 1 1^\top$ we recover a centralized variant of \ADAM~called (\CADAM). 
% The main steps of the Decentralized \ADAM~(\DADAM) are given in Algorithm~\ref{alg:1}. In \DADAM, at iteration $k$, each node $i$ updates sequences $\x_{i,k}$ and $\z_{i,k}$ based on the information received form its neighbouring nodes $\mathcal{N}_i$. Specifically, each node $i$ computes $\x_{i,k-\frac{1}{2}}$ and adaptive momentum direction $\sdi_{i,k}$ and then generates a new iterate $\z_{i,k}$ and $\x_{i,k}$ by choosing an appropriate step toward this direction. \DADAM{} is initialized at $\z_{i,0} = \x_{i,0}=0$ to keep the presentation of the algorithm and convergence analysis succinct. In general, any initialization can be selected for implementation purposes. 

\section{Convergence Analysis}\label{sec:CA}
This section provides non-asymptotic convergence rates for the \DADAM algorithm. The following assumptions are required to establish the results. 

\begin{assum}\label{assumption:function}
For all $\x, \y \in \mathbb{R}^d$, 
\begin{enumerate}
    \item  \label{assumption:g_unbi} $\E_{\bxi \sim \mathcal{D}}[\nabla f(\x,\bxi)] = \nabla F(\x)$.
    \item  \label{assumption:g_bounded} The function $f(\x,\bxi)$ has a $G_{\infty}$-bounded gradient, i.e., $\forall \ \bxi \sim \mathcal{D}$, it holds that $\|\nabla f(\x,\bxi)\|_{\infty} \leq G_{\infty} <\infty$. 
    \item  \label{assu:bounded_var}  The function $F(\cdot)$ has bounded variance, i.e., 
$$\mathbb{E}_{\bxi \sim \mathcal{D}}\left[\|\nabla f(\x,\bxi) -\nabla F(\x)\|^2\right] = \sigma^2<\infty.$$
    \item \label{assu:Lipshitz_gradient}The function $F(\cdot)$ is gradient Lipschitz continuous, i.e., $\exists \ L \; \textnormal{such that } 0\leq L < \infty$ and
    $$\|\nabla F(\x) - \nabla F(\y)\| \leq L\|\x-\y\|.$$
\end{enumerate}
\end{assum}

The above assumptions are fairly standard in the non-convex optimization literature \citep{bottou2018optimization}.
Further, Assumption~\ref{assumption:function}(\ref{assumption:g_bounded}) is slightly stronger than the assumption $\|\nabla f(\x,\bxi)\| \leq G_2$ used in the analysis of SGD . However, this assumption is crucial for the convergence analysis of adaptive methods in the non-convex setting and it has been widely considered in the literature \citep{zaheer2018adaptive,zhou2018convergence,liu2019towards,nazari2019dadam,nazari2020adaptive}.

\begin{assum}[Minty VI condition]
\label{assumption:bounded-spaceI} There exists $\x_*\in \mathbb{R}^d$ such that for any $\x \in \mathbb{R}^d$, 
\[
\langle \x-\x_*,  \nabla F(\x) \rangle \geq 0.
\]
\end{assum}

This assumption is commonly used in non-convex minimization problems due to its practicality~\citep{li2017convergence,kleinberg2018alternative,zhou2019sgd}. Subsequently, it was widely adopted and commonly used for min-max optimization problems as well~\citep{mertikopoulos2018optimistic,liu2019towards,liu2019decentralized,dang2015convergence,iusem2017extragradient, lin2018solving, wai2018multi,waiprovably}.  This assumption also has been commonly used in the literature for studying min-max problems that result from training GANs; see~\cite{liu2019decentralized, liu2019towards, mertikopoulos2018optimistic} and references therein.

\begin{assum}\label{assumption:bounded-space} 
We assume for the points $\x^*$ in Assumption~\ref{assumption:bounded-spaceI} and $\z_{i,k}$ generated by Algorithm~\ref{alg:1} that the following hold: $\|\x_*\| \leq \frac{D}{2}$ and $\|\z_{i,k}\| \leq \frac{D}{2}, \; \forall \ i,k$.
\end{assum}

Finally, for the topology of the mixing matrix $\W$, we require the following assumption.  
\begin{assum}[Mixing Matrix]\label{assumption:graph} 
The graph $\mathcal{G} = (V, E)$ is fixed and undirected; the associated matrix $\W$ of $\mathcal{G}$ is doubly stochastic; and 
\begin{equation}
\rho:= \max\{|\sigma_2(\W)|, |\sigma_n(\W)| \} <1.       
\end{equation}
\end{assum}
%The above assumption is weaker than the commonly used assumption in the literature where the mixing matrix is considered to be strongly connected~\citep{liu2019decentralized,}. 
Assumption D is common in the literature of decentralized algorithms; see, e.g., ~\citep{yuan2016convergence,nazari2019dadam}.

\begin{theorem}[Informal statement]\label{thm:revised_main}
Suppose Assumptions~\ref{assumption:function}--\ref{assumption:graph} hold. Let
\begin{align*}
& \beta_{1,t}=\beta_{1,1}\lambda^{t-1}, ~~~  \lambda\in(0,1), ~~~ \beta_{1,1} \leq \frac{1}{1+\sqrt{108 G_{\infty}/{MG_0^3}}}, ~~~ \{\beta_h\}_{h=1}^3 \in [0,1),~~ \\
& \eta \leq \sqrt{\frac{G_{0}^3}{72 L^2 G_{\infty}}}, ~~~\textnormal{and}~~~ G_{0}^2 \leq \|\sVtori_{i,0}\|_{\infty} \leq G_{\infty}^2. 
\end{align*}
Then, the iterates of \DADAM satisfy
\begin{align*}
&\frac{1}{N}\sum_{k = 1}^N \E\|\nabla F (\bar\z_k)\|^2  \leq  \frac{B_1}{N} + \frac{B_2 }{MN} \sum_{k = 1}^N \frac{\sigma^2}{m_k} + \frac{B_3 M\rho^t}{1-\rho^t},
\end{align*}
where $B_1, B_2$ and $B_3$ are constants obtained by~\eqref{eq:revised_const_global}.
\end{theorem}
Theorem~\ref{thm:revised_main} summarizes the main result of this paper and provides an upper bound for the average square norm of the gradient. Its significance for the problem under consideration is manifested in the following corollary.

\begin{corollary}\label{revised_cor:1}
Under the Assumptions of Theorem~\ref{thm:revised_main}, if $t \geq \log_{\frac{1}{\rho}} \left(1 + C\right), N \sim \tilde  {\mathcal{O}}(\epsilon^{-2})$, then for all settings in the following table, there exists an iterate of Algorithm~\ref{alg:1} that is an $\epsilon$-SFNE point of the game defined in~\eqref{eq:main_game}.

\begin{table}[H]
  \centering

  \begin{tabular}{|P{2.5cm}|P{2.5cm}|P{2.5cm}|}
    \hline
    $m_i    $ & $M$ & $C$ \\ \hline
    $ \mathcal{O}(\epsilon^{-2})$             & $ \mathcal{O}(1)$    & $ \mathcal{O}(\epsilon^{-2})$   \\ \hline
    $ = i + 1$ & $  \mathcal{O}(1)$  &  $\mathcal{O}(\epsilon^{-2})$   \\ \hline
   $\mathcal{O}(1)$  & $ \mathcal{O}(\epsilon^{-2})$ & $ \mathcal{O}(\epsilon^{-4})$  \\ \hline
  \end{tabular}
\caption{Required complexity to obtain $\epsilon$-SFNE in different settings.  }\label{tab1}
\end{table}
\end{corollary}

\section{Performance Evaluation of \DADAM}\label{sec:NS}

This section presents a comprehensive empirical evaluation of \DADAM based on a representative suite of decentralized data sets and modeling tasks.
The objective is to both validate the theoretical results, and also evaluate the performance of the proposed method compared to benchmark methods for real data sets. To that end, we conduct simulations of decentralized learning algorithms for a synthetic data experiment, and also use the algorithms for training GANs.
In both evaluation tasks, we assume that $M$ nodes communicate through a ring topology; i.e., $\W \in \mathbb{R}^{M \times M}$ is defined as
\begin{align*}
[\W]_{i,j} = 
 \begin{cases}
   1/3 &\quad\text{if } |i-j| \leq 1 \; \text{or} \; |i-j| = M-1,\\
   0 &\quad\text{o.w}. \\
 \end{cases}
\end{align*}
This is a commonly used topology in the decentralized optimization literature~\citep{zhang2019distributed,patarasuk2009bandwidth,zhang2019highly,zhang2020improving}.
For the centralized variants of the methods, we consider $\W = \frac{1}{M} \bf 1 \bf 1^\intercal$.
As previously mentioned, the main focus of the paper is to study the performance of  Algorithm~\ref{alg:1}. However, in this section, we will also evaluate the performance of  Algorithm~\ref{alg:0} that works in a single machine set-up. A detailed theoretical analysis and numerical experiments for the latter algorithm, see~\cite{barazandeh2021solving}.
\subsection{A Synthetic Data Experiment}
To better understand the challenges of solving the type of min-max problems posited in~\eqref{eq:main_game_s}, consider the following function
\begin{equation}\label{eq:numerical}
  f(\theta,\alpha) =
    \begin{cases}
      c(\theta-\alpha) + (\theta^2-\alpha^2)+ k \theta \alpha,  &  \text{w.p\;} \frac{1}{3},  \\
      (\theta-\alpha) + (\theta^2-\alpha^2) + k \theta \alpha, & \text{w.p\;} \frac{2}{3},
    \end{cases}       
\end{equation}
with $c \geq 1$ and $k \geq 0$. It can easily be shown that
\begin{align*}
F(\theta,\alpha) = \frac{c+2}{3} (\theta-\alpha)+ (\theta^2-\alpha^2)+ k \theta \alpha,    
\end{align*}
which has the unique FNE point $(\theta^*,\alpha^*) =-(c+2)/(3k^2 + 12) (2-k,2+k)$. It can be seen that this is a strongly-convex-strongly-concave, since $\nabla^2_{\theta} F(\theta,\alpha) = 2\I \succ 0$ and $\nabla^2_{\alpha}F(\theta,\alpha) = -2\I \prec 0$.  
In this experiment, we consider $M = 5$ nodes that are connected to each other using the ring topology previously defined. We use a decentralized version of ADAM called Decentralized Parallel Optimistic Adam (DP-OAdam)~\cite{liu2019decentralized, daskalakis2017training} to find the FNE of the above problem and compare its performance with our proposed Algorithm~\ref{alg:1}. We use the following two metrics to evaluate the performance of the algorithms; $e_k = \frac{\|\bar\z_k - \z_*\|}{\|\z_*\|}$ and $\mathcal{R}_k = \frac{1}{k} \sum_{i = 1}^k \| \nabla F(\bar\z_k)\|^2$ that are the error rate and average squared norm of the gradients at iteration $k$, respectively. We set the parameters at $c = 1010, k = 0.01, N = 10^{7}, \eta = 10^{-2}, \beta_1 = 0, \beta_2 = 1/(1+c^2)$ and $\beta_3 = 0.1$. All other parameters are initialized at zero. Figure~\ref{fig:first_simulation} shows the result of the simulated experiment. It can be seen that \DADAM finds the unique FNE point while DP-OAdam diverges from the desired point and is unable to find the FNE of this simple strongly-convex-strongly-concave game. This phenomenon is common in approaches that simultaneously update the minimization and maximization parameters (more details can be found in~\citep{razaviyayn2020nonconvex} and references therein).    
\begin{figure}[t]
%\begin{center}
\begin{subfigure}{.5\textwidth}
  \centering
  % include first image
  \includegraphics[width=.8\linewidth]{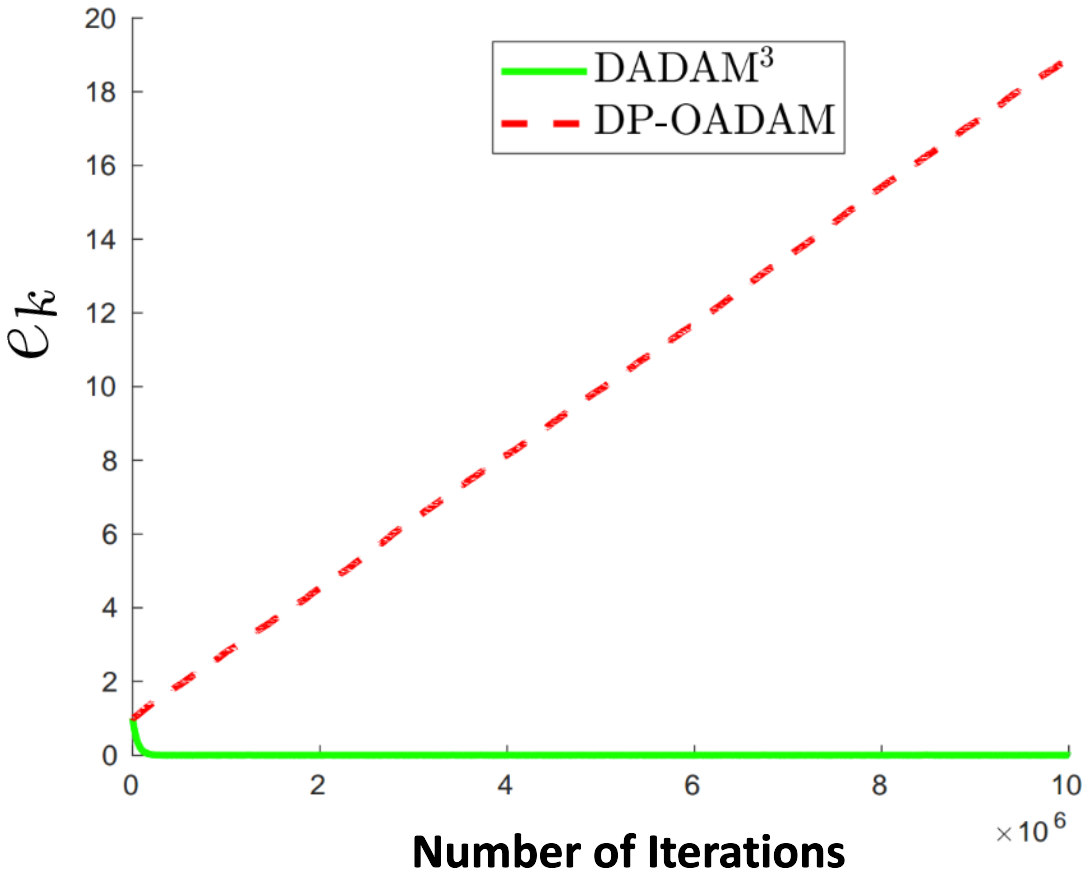}  
  \caption{Error rate}
  \label{fig:sub-first}
\end{subfigure}
\begin{subfigure}{.5\textwidth}
  \centering
  % include second image
  \includegraphics[width=.8\linewidth]{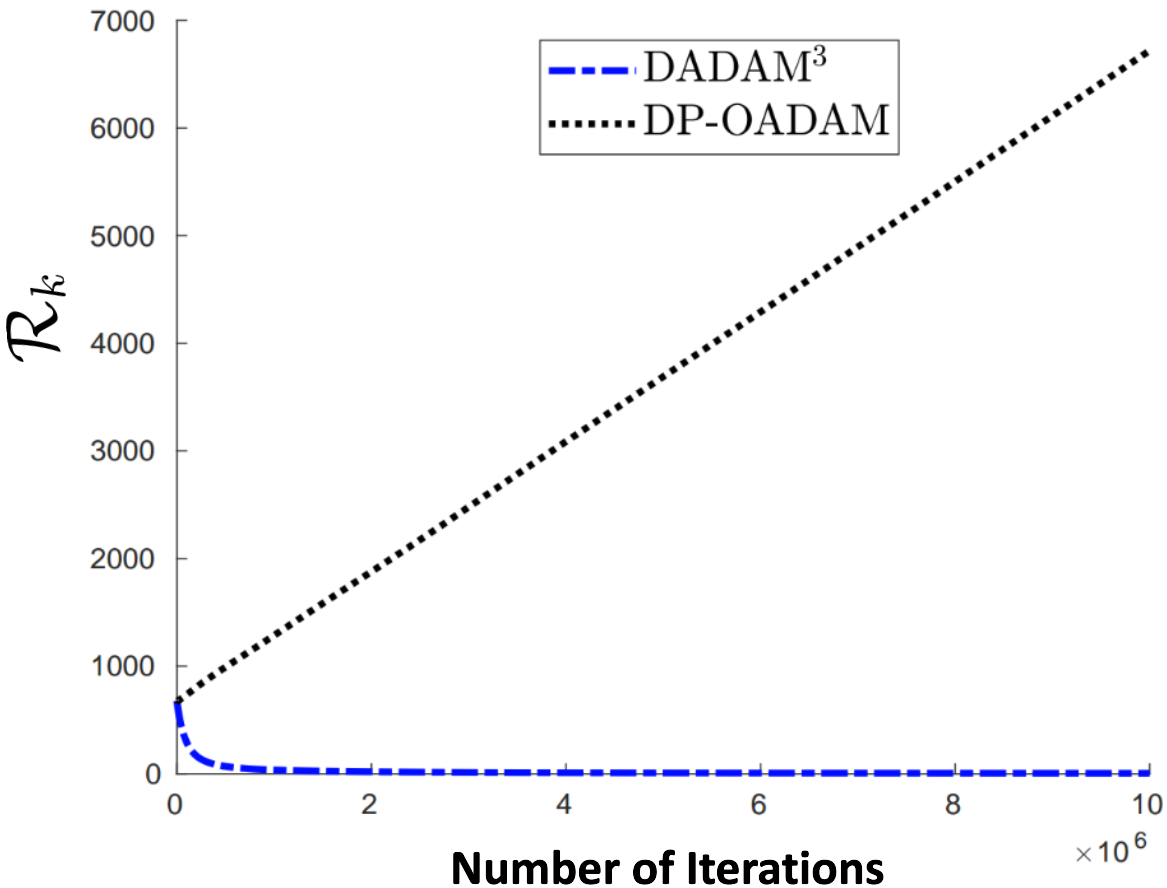}  
  \caption{Average norm-squared of gradient}
  \label{fig:sub-second}
\end{subfigure}
\caption{Simulation result for synthetic data experiment; Simple generalization of adaptive methods for decentralized problem might diverge in practice.}
\label{fig:first_simulation}
%\end{center}
\end{figure}

% \begin{figure}[h]
% \centering
% \includegraphics[scale=0.25]{images/twoinone.png}
% \caption{Left/Right $y$-axis: Error rate, $e_t$ / Average norm of gradient, $\mathcal{R}_k$. \SSADAM~ misses the unique FNE point.}

% \end{figure}
%
\subsection{Training GANs using benchmark data sets}
In this section, we compare the empirical performance of the proposed algorithms with benchmark methods for training a Wasserstein GAN Gradient Penalty network~(WGAN-GP)~\cite{gulrajani2017improved} to generate samples from the MNIST, CIFAR-10, CIFAR-100, Large-scale CelebFaces Attributes (CelebA) and ImageNet data sets. A brief description of each data set is provided next.

\textbf{MNIST\footnote{http://yann.lecun.com/exdb/mnist/} :} This data set contains about 70,000 samples of handwritten digits. We have used this data set to evaluate the performance of the proposed Algorithm~\ref{alg:1} using $M = 5$ nodes. Note that the MNIST images are of size $28 \times 28$ and we have scaled them to size $32 \times 32$ in order to use the same architecture as for the two CIFAR data sets.  

\textbf{CIFAR-10\footnote{https://www.cs.toronto.edu/~kriz/cifar.html}:} This data set consists of 60000 color images assigned to 10 different classes. We have used this data set to evaluate the performance of the proposed Algorithm~\ref{alg:0} and Algorithm~\ref{alg:1} using $M = 5$ nodes. 
%For detailed experiment of Algorithm~\ref{alg:0} for this dataset refer to \cite{barazandeh2021solving}.

\textbf{CIFAR-100\footnote{https://www.cs.toronto.edu/~kriz/cifar.html}:} This data set is similar to CIFAR-10. However, it consists of 100 classes containing 600 images each (total of 60000 color images). We have used this data set to evaluate the performance of the proposed Algorithm~\ref{alg:1} using $M = 15$ nodes.

\textbf{Imagenette\footnote{https://github.com/fastai/imagenette}:} The original ImageNet\footnote{http://www.image-net.org/} data set contains more than 4 million images in more than 20,000 categories. In almost all of the benchmark implementations~\cite{odena2017conditional, zhang2019self}, a subset of ImageNet is used for the purpose of training GANs. We will use the commonly used Imagenette that is a subset of 10 classes from Imagenet. We use this data set to train GANs using $M = 15$ nodes. The images are also scaled to size $64 \times 64$.

\textbf{CelebA\footnote{http://mmlab.ie.cuhk.edu.hk/projects/CelebA.html}:} This is a large scale data set that contains more than $200,000$ celebrity images. After reshaping images to $64 \times 64$,  we use this data set to evaluate the performance of the proposed Algorithm~\ref{alg:0}.   

Images in all of the above data sets are normalized to have mean intensity and standard deviation equal to 0.5 before feeding them into the model.

% their size as well as their access link.

% \begin{table}[H]
%   \centering
%   {\color{red}
%   \begin{tabular}{|P{2.5cm}|P{7.5cm}|P{2.5cm}|}
%     \hline
%     Dataset & Access Link & Number of Samples \\ \hline
%   MNIST       & http://yann.lecun.com/exdb/mnist/ & 70,000\\ \hline
%     CIFAR-10  &  https://www.cs.toronto.edu/~kriz/cifar.html & 60000 \\ \hline
%   CIFAR-100  & $ \mathcal{O}(\epsilon^{-2})$ & $ \mathcal{O}(\epsilon^{-4})$  \\ \hline
%   ImageNet  & $ \mathcal{O}(\epsilon^{-2})$ & $ \mathcal{O}(\epsilon^{-4})$  \\ \hline
%   CelebA  & $ \mathcal{O}(\epsilon^{-2})$ & $ \mathcal{O}(\epsilon^{-4})$  \\ \hline
%   \end{tabular}
%   }
%   \newline\newline
% \caption{{\color{red} Dataset used for training GANs.  }}
% \end{table}
% Our network comprises again of $M = 5$ nodes and the data are split equally among them. 

\subsection{Experimental Set-up}
\paragraph{Generator and Discriminator model for MNIST, CIFAR-10 and CIFAR-100} The generator's network consists of an input layer, 2 hidden layers and an output layer. Each of the input and hidden layers consist of a transposed convolution layer followed by batch normalization and ReLU activation. The output layer is a transposed convolution layer with hyperbolic tangent activation. 
The network for the discriminator also has an input layer, 2 hidden layers and an output layer. All of the input and hidden layers are convolutional ones followed by instance normalization and Leaky ReLU activation with slope~$0.2$. The design for both networks is summarized in Tables~\ref{tab:generator_arch} and~\ref{tab:discriminator_arch}.

\paragraph{Generator and Discriminator model for CelebA and Imagenette} The generator's network consists of an input layer, 3 hidden layers and an output layer. Each of the input and hidden layers consist of a transposed convolution layer followed by batch normalization and ReLU activation. The output layer is a transposed convolution layer with hyperbolic tangent activation. 
The network for the discriminator also has an input layer, 3 hidden layers and an output layer. All of the input and hidden layers are convolutional ones followed by instance normalization and Leaky ReLU activation with slope~$0.2$. The design for both networks is summarized in Tables~\ref{tab:generator_arch2} and~\ref{tab:discriminator_arch2}.

%%%%%%%%%%%%%%%%

\begin{table}[h]
\centering
\resizebox{\textwidth}{!}{%
\begin{tabular}{@{}cc@{}}
\toprule
\multirow{2}{*}{Layer Type} & Shape \\ 
 & (Ch\_in, Ch\_out, kernel, stride, padding) \\ \midrule
ConvTranspose $+$ BatchNorm $+$ ReLU & $(100, 1024, 4, 1, 0)$ \\
ConvTranspose $+$ BatchNorm $+$ ReLU & $(1024, 512, 4, 2, 1)$ \\
ConvTranspose $+$ BatchNorm $+$ ReLU & $(512, 256, 4, 2, 1)$ \\
ConvTranspose $+$ Tanh & $(256, C, 4, 2, 1)$ \\ \bottomrule
\end{tabular}
}

\caption{Architecture of the Generator. $C = 1$ for MNIST and $C = 3$ for CIFAR-10 and CIFAR-100.}
\label{tab:generator_arch}
\end{table}

\begin{table}[h]
\centering
\resizebox{\textwidth}{!}{%

\begin{tabular}{@{}cc@{}}
\toprule
\multirow{2}{*}{Layer Type} & Shape \\ 
 & (Ch\_in, Ch\_out, kernel, stride, padding) \\ \midrule
Conv $+$ InstanceNorm $+$ LeakyReLU & $(C, 256, 4, 2, 1)$ \\
Conv $+$ InstanceNorm $+$ LeakyReLU & $(256, 512, 4, 2, 1)$ \\
Conv $+$ InstanceNorm $+$ LeakyReLU & $(512, 1024, 4, 2, 1)$ \\
Conv & $(1024, 1, 4, 1, 0)$ \\ \bottomrule
\end{tabular}
}

\caption{Architecture of the Discriminator. $C = 1$ for MNIST and $C = 3$ for CIFAR-10 and CIFAR-100.}
\label{tab:discriminator_arch}
\end{table}

\begin{table}[h]
\centering
\resizebox{\textwidth}{!}{%
\begin{tabular}{@{}cc@{}}
\toprule
\multirow{2}{*}{Layer Type} & Shape \\ 
 & (Ch\_in, Ch\_out, kernel, stride, padding) \\ \midrule
ConvTranspose $+$ BatchNorm $+$ ReLU & $(100, 512, 4, 1, 0)$ \\
ConvTranspose $+$ BatchNorm $+$ ReLU & $(512, 256, 4, 2, 1)$ \\
ConvTranspose $+$ BatchNorm $+$ ReLU & $(256, 128, 4, 2, 1)$ \\
ConvTranspose $+$ BatchNorm $+$ ReLU & $(128, 64, 4, 2, 1)$ \\
ConvTranspose $+$ Tanh & $(64, C, 4, 2, 1)$ \\ \bottomrule
\end{tabular}
}
\caption{Architecture of the Generator. $C = 3$ for Celeba and Imagenet.}
\label{tab:generator_arch2}
\end{table}

\begin{table}[h]
\centering
\resizebox{\textwidth}{!}{%

\begin{tabular}{@{}cc@{}}
\toprule
\multirow{2}{*}{Layer Type} & Shape \\ 
 & (Ch\_in, Ch\_out, kernel, stride, padding) \\ \midrule
Conv $+$ InstanceNorm $+$ LeakyReLU & $(C, 64, 4, 2, 1)$ \\
Conv $+$ InstanceNorm $+$ LeakyReLU & $(64, 128, 4, 2, 1)$ \\
Conv $+$ InstanceNorm $+$ LeakyReLU & $(128, 256, 4, 2, 1)$ \\
Conv $+$ InstanceNorm $+$ LeakyReLU & $(256, 512, 4, 2, 1)$ \\
Conv & $(512, 1, 4, 1, 0)$ \\ \bottomrule
\end{tabular}
}

\caption{Architecture of the Discriminator. $C = 3$ for Celeba and Imagenet.}
\label{tab:discriminator_arch2}
\end{table}

%%%%%%%%%%%%%%%

% \begin{itemize}
%   \item Generator: [ConvTranspose2d(100, 256, 4, 1, 0),~BatchNorm2d(256), ReLU, ConvTranspose2d(256, 128, 4, 2, 1), BatchNorm2d(128), ReLU, ConvTranspose2d(128, 64, 4, 2, 1), BatchNorm2d(64), ReLU, ConvTranspose2d(64, 1, 4, 2, 1), Tanh ]
%   \item Discriminator: [Conv2d(1, 32, 4, 2, 1), InstanceNorm2d(32), LeakyReLU(0.2), Conv2d(32, 64, 4, 2, 1), InstanceNorm2d(64), LeakyReLU(0.2), Conv2d(64, 128, 4, 2, 1), InstanceNorm2d(128), LeakyReLU(0.2), Conv2d(128,1,4,1,0) ]
% \end{itemize}

\paragraph{Implementation} We implement all algorithms in PyTorch and an open-source implementation of the algorithms and the code for the experiments will be made available. The training parameters used in all experiments are summarized in Table~\ref{tab:para_training}.
% We set the step-size to be $\eta = 0.5 \times 10^{-3}$, $\beta_1 = 0.5$ and $\beta_2 = 1$. The batch size is also set to be $64$. Also, the experiment is run for a total of 30000 iterations. 

\begin{table}[h]
\centering
\begin{tabular}{@{}ll@{}}
\toprule
Parameter     &  Value                 \\ \midrule
Learning Rate & $5*10^{-5}$   \\ 
$\beta_1$ & $0.5$\\
$\beta_2$ & $0.999$\\
$\beta_3$& $0.5$\\
Batch-size    & 64   \\
Training Iterations & $30000$\\
\bottomrule
\end{tabular}
\caption{Training Parameters.}
\label{tab:para_training}
\end{table}

\subsection{Results}
In this section, we evaluate the performance of our proposed algorithms in the task of training GANs. For the case of a single computing node, we compare Algorithm~\ref{alg:0} with the recently proposed OAdagrad~\cite{liu2019towards} which is an adaptive method. 

For the multiple computing nodes case, we use the proposed decentralized algorithm, \DADAM, and compare it with its centralized variant~\CADAM. Additionally, we compare the proposed algorithm with the Decentralized Optimistic Stochastic Gradient~(\DSGD) and its centralized variant called~\CSGD.

In this evaluation, we run the above listed algorithms for training a WGAN-GP and output the generated images and report their inception scores. The reported results are based on the average performance across all nodes. 
\noindent

\textit{Comparisons based on the inception score:} This score is one of the most commonly used measures for evaluating the quality of images generated by GANs~\cite{salimans2016improved}. A high value corresponds to images of higher quality and also more realistic to the human eye~\cite{barratt2018note} .

\begin{figure}[t]
%\begin{center}
\begin{subfigure}{.5\textwidth}
  \centering
  % include first image
  \includegraphics[width=.85\linewidth]{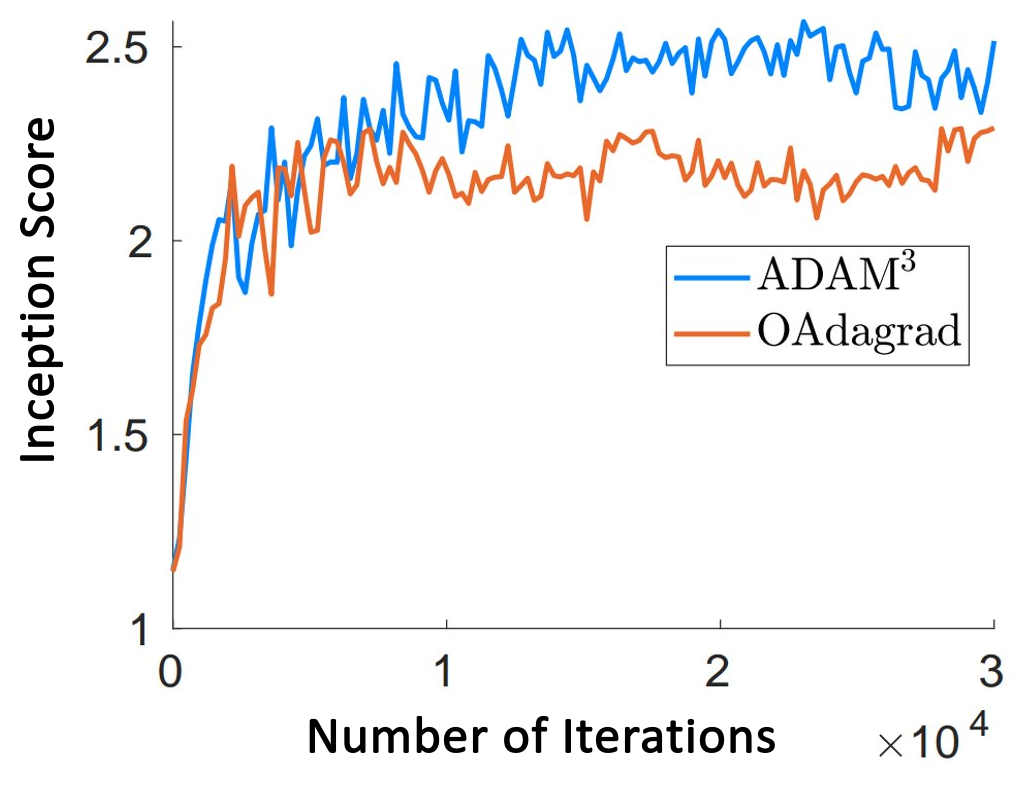}  
  \caption{Inception Score-CelebA}
  \label{fig:sub-first}
\end{subfigure}
\begin{subfigure}{.5\textwidth}
  \centering
  % include second image
  \includegraphics[width=.85\linewidth]{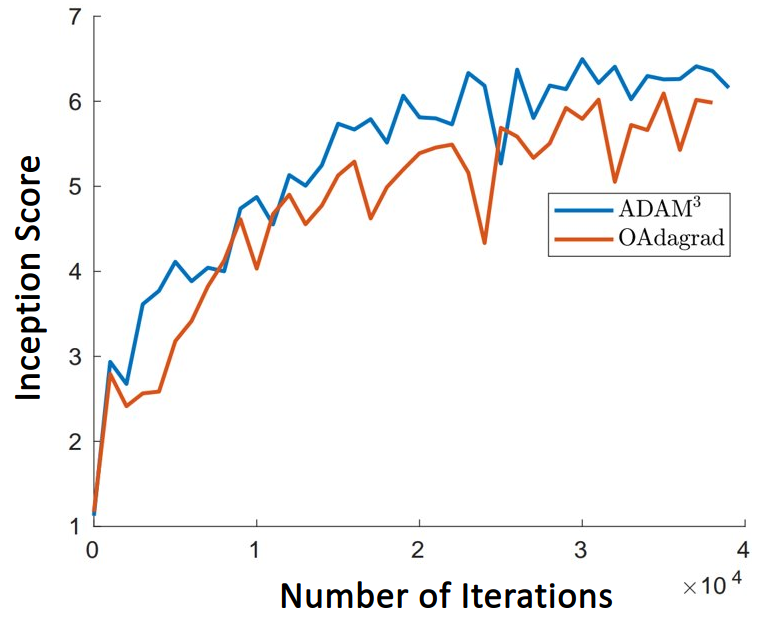}  
  \caption{Inception Score-CIFAR-10~\cite{barazandeh2021solving}}
  \label{fig:sub-second}
\end{subfigure}
\caption{Inception Score for CelebA(left) and CIFAR-10 (right) in Single Computing Node Experiment; }
\label{fig:result_single_machine}
%\end{center}
\end{figure}

\begin{figure}[t]
%\begin{center}
\begin{subfigure}{.5\textwidth}
  \centering
  % include first image
  \includegraphics[width=.85\linewidth]{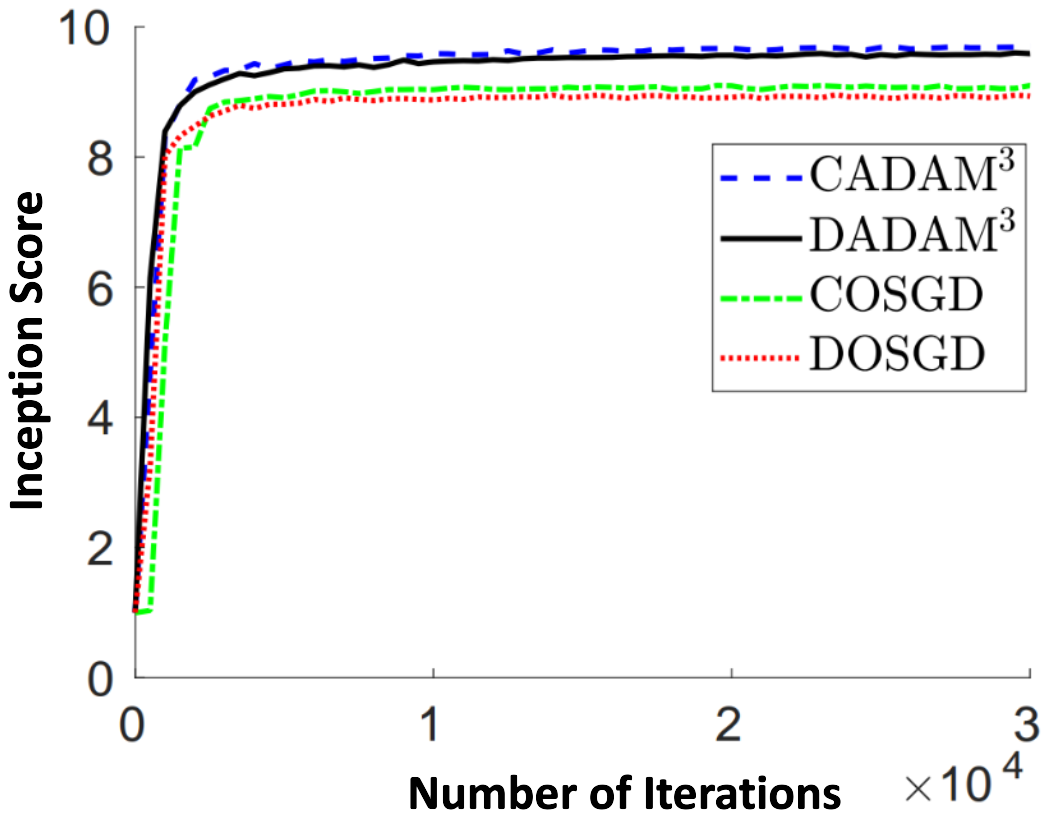}  
  \caption{Inception Score-MNIST}
  \label{fig:sub-first}
\end{subfigure}
\begin{subfigure}{.5\textwidth}
  \centering
  % include second image
  \includegraphics[width=.85\linewidth]{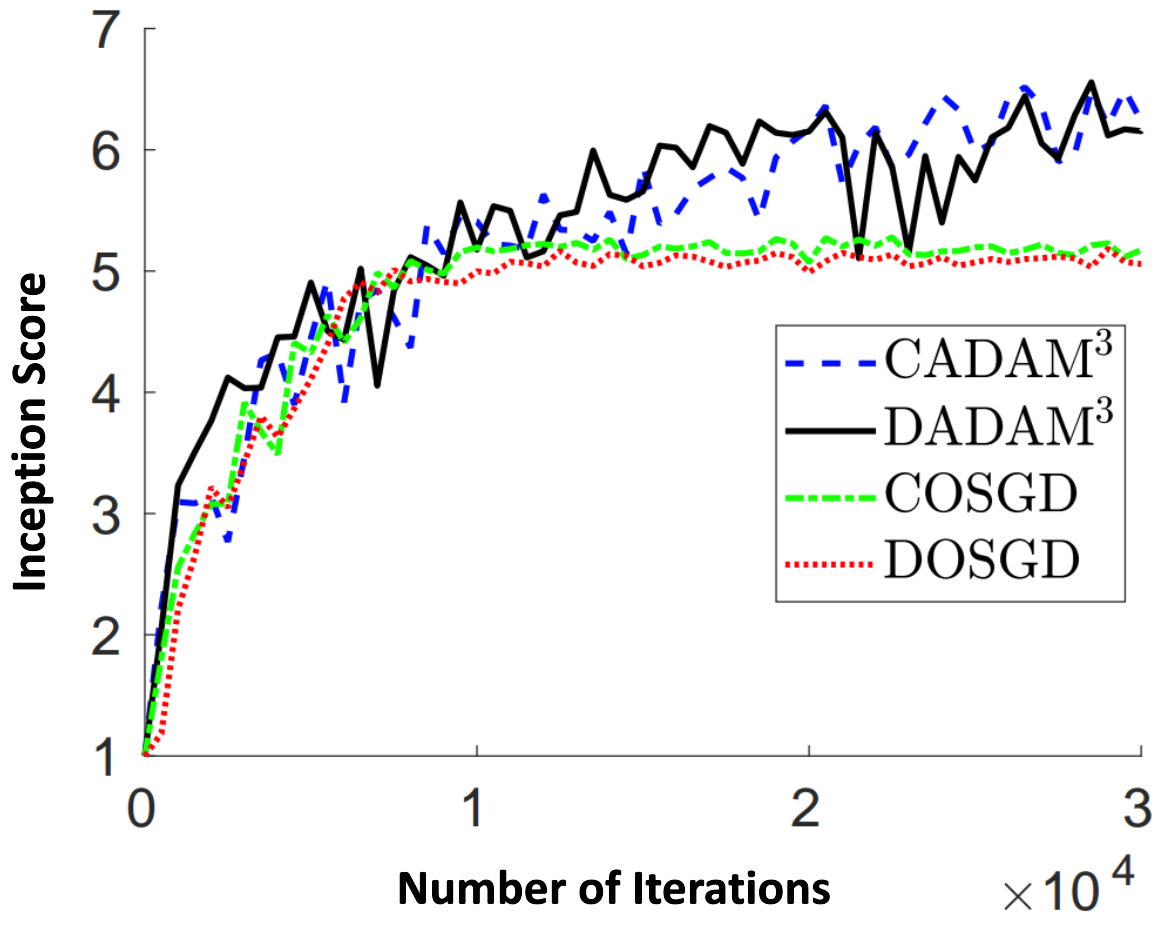}  
  \caption{Inception Score-CIFAR10}
  \label{fig:sub-second}
\end{subfigure}
\caption{ Inception Score for MNIST~(left) and CIFAR-10~(right)  using $M=5$ nodes; }
\label{fig:result}
%\end{center}
\end{figure}

\begin{figure}[t]
%\begin{center}
\begin{subfigure}{.5\textwidth}
  \centering
  % include first image
  \includegraphics[width=.85\linewidth]{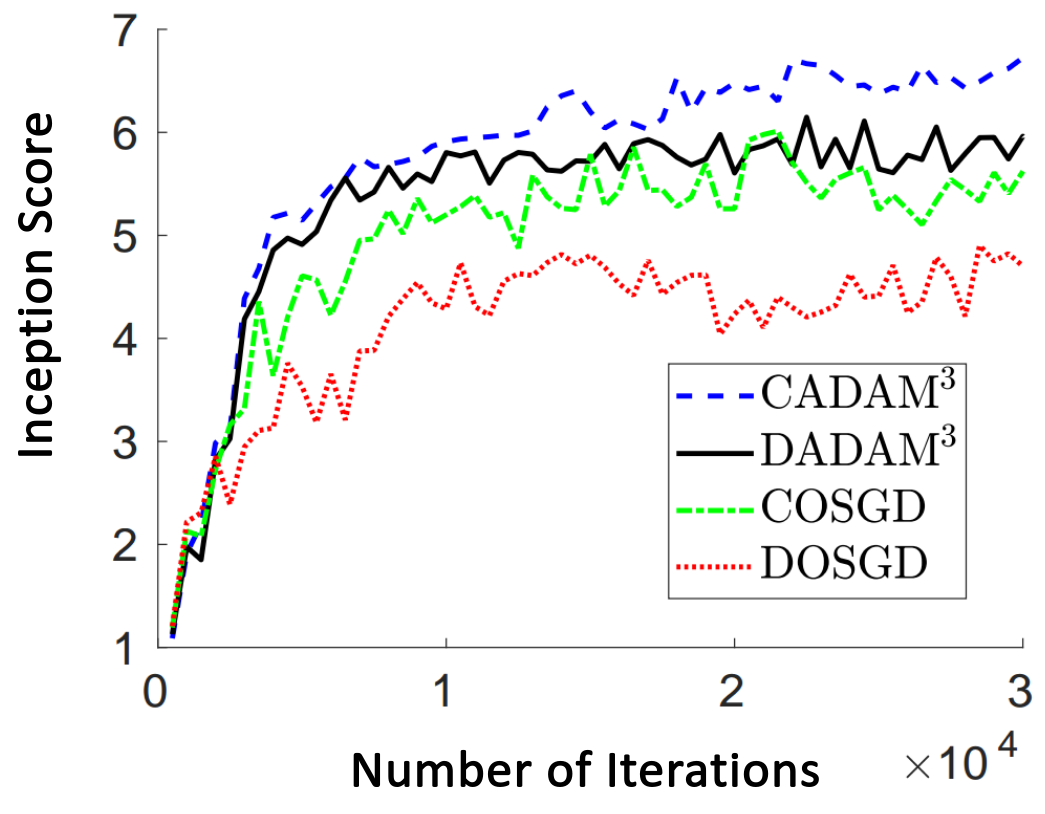}  
  \caption{Inception Score-Imagenette}
  \label{fig:sub-first1}
\end{subfigure}
\begin{subfigure}{.5\textwidth}
  \centering
  % include second image
  \includegraphics[width=.85\linewidth]{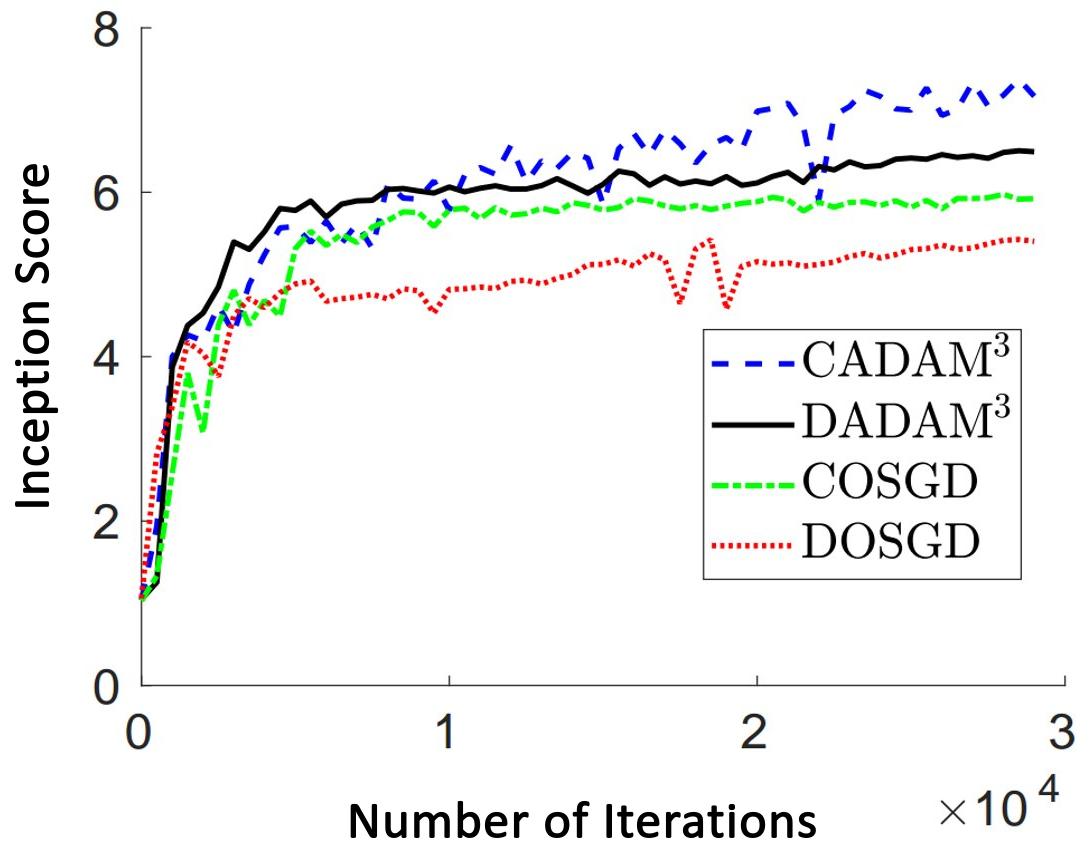}  
  \caption{Inception Score-CIFAR-100}
  \label{fig:sub-second1}
\end{subfigure}
\caption{Inception Score for Imagenette~(left) and CIFAR-100~(right) using $M=15$ nodes;}
\label{fig:result_cifar_100_imagnet}
%\end{center}
\end{figure}

To calculate the inception score, we need to have a high-accuracy pre-trained classifier that will output the conditional probability of a given input image, i.e., the probability that an image belongs to any of the prespecified classes. Specifically, the logarithm of the inception score is calculated by
$$
\log~\text{Inception Score} = \E_{\x~\sim p_g} \KL(p(y|\x)||p(\x) ),
$$
where $p(y|\x)$ denotes the conditional probability,  $p(\x) = \int p(y|\x) p_g(\x)$  the marginal probability,  $\KL(\cdot,\cdot)$ the Kullback-Leibler divergence between two distributions, and $p_g$ is the distribution of the images generated by $g$.

Figure~\ref{fig:result_single_machine}, panels (a)-(b) show the inception scores for CelebA and CIFAR-10 images generated by the GANs trained by Algorithm~\ref{alg:0} and OAdagrad. Figure~\ref{fig:result}, panels (a)-(b) and Figure~\ref{fig:result_cifar_100_imagnet}, panels (a)-(b) show plots of the inception scores for MNSIT and CIFAR-10 using $M = 5$ nodes and CIFAR-100 and Imagenette using $M = 15$ nodes, respectively, using the proposed Algorithm~\ref{alg:1}. We also provide samples of generated images for all of the competing methods in the Supplement. It can be seen that Algorithm~\ref{alg:0} outperforms OAdagrad. Further, \DADAM and \CADAM are able to generate high quality images and reach a higher inception score more quickly compared to~\textsc{Dsgd} and~\textsc{Csgd}. This clearly illustrates that adaptive momentum methods exhibit superior performance compared to stochastic gradient methods. On the other hand, comparing decentralized methods with their centralized variants shows that the latter are capable of generating high quality images, but with higher communication cost and computational complexity.

\section{Conclusion}
This paper develops a new decentralized adaptive momentum method for solving min-max optimization problems. The proposed \DADAM algorithm enables data parallelization, as well as decentralized computation. We also derive the non-asymptotic rate of convergence for the proposed algorithm for a class of non-convex-non-concave games and evaluate its empirical performance by using it for training GANs. The experimental results on GANs illustrates the effectiveness of the proposed algorithm in practice.

\bibliography{mybibfile}

\newpage
\section*{Supplement to "Adaptive Momentum Methods for Min-Max Optimization''}

The proofs of all technical results, and
 key technical lemmas are presented next.
 We first introduce notation used throughout this section.

\textbf{Notation.} Throughout the Supplement, $\mathbb{R}^d$ denotes the $d$-dimensional real Euclidean space. $\II_d$ is used to denote the $ d \times d$ identity matrix.  For any pair of vectors $\x,\z\in \mathbb{R}^d,\,\, \langle \x , \z \rangle$ indicates the standard Euclidean inner product.  We denote the element in the $i^{th}$ row and $j^{th}$ column of matrix $\X$ by $[\X]_{ij}$.
We denote the $\ell_1$ norm by $\|\X\|_1 = \sum_{ij}|[\X]_{ij}|$, the maximum element of the matrix by $\|\X\|_{\max} = \max_{ij}|[\X]_{ij}|$, and the Frobenius norm by $\|\X\|= \sqrt{\sum_{ij}| [\X]_{ij}|^2}$, respectively. The above norms reduce to the vector norms if $\X$ is a vector. We also define the following notation in addition to~\eqref{eq:nota} that will be used in the proofs,
\begin{align}\label{eq:revised_nota2}
\nonumber
& \Vt_{k} = \left[\sVtori_{1,k}^{-\frac{1}{2}},\cdots, \sVtori_{M,k}^{-\frac{1}{2}}\right],\rg_k = \left[\srg_{1,k},\cdots, \srg_{M,s}\right],
\\
& \eg_k =\left[\seg_{1,k};,\cdots, \seg_{M,k}\right], \bg_{k}= \left[\sbg_k,\cdots, \sbg_k\right],\beps_k = \eg_k - \rg_k.   
\end{align}

In the above notation, $\srg_{i,k}$ is the gradient value calculated at point $\z_{i,k}$ , $\seg_{i,k}$ is its estimated value and $\sbg_k$ is the gradient value calculated at the averaged point $\bar\z_k.$  

Using the above notation, we can then rewrite the update rule of $\X_k$ and $\Z_K$ for all the nodes in Algorithm~\ref{alg:1} together as

\begin{equation*}
\begin{aligned}[c]
\Z_1 &= (\X_0 - \eta \di_0)\W^t, 
\\
\Z_2 &= (\X_1 - \eta \di_1)\W^t 
\\ &= \X_0\W^{2t} - \eta \di_1 \W^{2t} - \eta  \di_1 \W^t,\\
&\quad  \vdots \\
\Z_k &= \X_0\W^{tk}-\eta\sum_{s=1}^{k-1} \di_s\W^{t(k-s+1)} -\eta \di_{k-1}\W^t,
\end{aligned}
\qquad
\begin{aligned}[c]
\X_1 &= (\X_0-\eta \di_1)\W^t,
\\
\X_2 &= (\X_1-\eta \di_2)\W^t 
\\ &= \X_0 \W^{2t}-\eta \di_1\W^{2t} - \eta  \di_2 \W^t \\
&\quad  \vdots \\
\X_k &= \X_0\W^{tk} -\eta\sum_{s=1}^{k} \di_{s}\W^{t(k-s+1)}. 
\end{aligned}
\end{equation*}

Next, by assuming $\X_k = \Z_k = \di_k = \M_k = 0, \ \forall \ k \leq 0$, we have
\begin{align}\label{eq:update_zx}
\nonumber 
\Z_k &= -\eta\sum_{s=1}^{k} \di_{s-1}\W^{t(k-s+2)} -\eta \di_{k-1}\W^t,\\ 
\X_k &= -\eta\sum_{s=1}^{k} \di_s \W^{t(k-s+1)}.
\end{align}
We will also use the value of $\bar\z_k$ in a matrix form in the proofs given by
\begin{align*}
\bar\z_k &=  \frac{1}{M} \Z_k \mathbf{1}_M = \frac{1}{M} \big(-\eta\sum_{s=1}^{k}\di_{s-1}\W^{t(k-s+2)}-\eta \di_{k-1}\W^t\big)\mathbf{1}_M
\\
& = -\frac{\eta}{M}\sum_{s=1}^{k}\di_{s-1}\mathbf{1}_M -\frac{\eta}{M} \di_{k-1}\mathbf{1}_M,
\end{align*}
where the last equality is due to Assumption~\ref{assumption:graph} and the doubly stochastic nature of the mixing matrix $\W$.

It can easily be seen that $\z_{i,k}$ is the $i^{th}$ column of matrix $\Z_k$, i.e.,
\begin{align*}
\z_{i,k}  &=  \Z_k \mathbf{e}_i =  \left(-\eta\sum_{s=1}^{k}\di_{s-1}\W^{t(k-s+2)}-\eta \di_{k-1}\W^t\right)\mathbf{e}_i, 
\end{align*}
with $\e_i = [0,  \ldots , 1, \cdots , 0]^\top$ where 1 is appearing at the $i^{th}$ coordinate. Putting both of them together, we get
\begin{align}\label{eq:updatez}
\nonumber
\bar\z_k &=  - \frac{\eta}{M}\sum_{s=1}^{k}\di_{s-1}\mathbf{1}_M - \frac{\eta}{M} \di_{k-1}\mathbf{1}_M, \quad \textnormal{and}
\\
\z_{i,k} &= -\eta\sum_{s=1}^{k}\di_{s-1}\W^{(k-s+2)}\mathbf{e}_i-\eta \di_{k-1}\W^t\mathbf{e}_i.   
\end{align}
% We also use $\circ$ and $\otimes$ to represent the Hadamard and Kronecker product, respectively.
Additionally, the update rule for $\di_k$ in Algorithm~\ref{alg:1} can be written as

\begin{align}\label{eq:update_D}
\di_{k} & = \beta_{1,k} \Vt_{k} \circ \M_{k-1} + (1-\beta_{1,k})\Vt_{k} \circ \eg_{k}.
\end{align}

\section{Technical Lemmas}
%We collect the technical lemmas with their proofs in this sections. They are widely used in the proof of the major theorems. 

Next, we explore some basics properties of the Hadamard product in Euclidean space.

\begin{lemma}\label{lem:hadamard} 
Let $\textbf{A}$ and $\textbf{B}$ be $H \times M$ matrices with entries in $\mathbb{R}$. Then, 

\begin{minipage}{0.45\textwidth}
\begin{enumerate}[label=(\roman*)]
\item \label{itm:one}$\|\sum\limits_{i = 1}^M \textbf{a}_{i}\|^2\leq M \sum\limits_{i = 1}^M \|\textbf{a}_i\|^2$;
 \item \label{itm:seven} $\|\textbf{A} \circ \textbf{B}\|_F \leq \|\textbf{A}\|_{\max} \|\textbf{B}\|_{1,1}$;
\item \label{itm:eight} $\|\textbf{A} \circ \textbf{B}\|_F \leq \|\textbf{A}\|_{\max} \|\textbf{B}\|_F$.
 \item  \label{itm:four} $ \|\textbf{A}\circ \textbf{B}\|_2 \leq \|\textbf{A}\|_2 \|\textbf{B}\|_2$;
\end{enumerate}
\end{minipage}%
\hfill
\begin{minipage}{0.65\textwidth}
\begin{enumerate}[label=(\roman*)]
 \setcounter{enumi}{3}
 \item \label{itm:five} $\|\textbf{A} \circ \textbf{B}\|_{F} \leq \|\textbf{A}\|_F \|\textbf{B}\|_F$;
\item \label{itm:two} $\left\Vert \frac{1}{M} (\textbf{A} \circ \textbf{B})\bm{1}_M \right\Vert^2 \leq \frac{1}{M} \|\textbf{A}\|^2_{\max}   \Vert \textbf{B}\Vert_F^2$;
\item \label{itm:three}$\left\Vert \frac{1}{M} (\textbf{A} \circ \textbf{B})\bm{1}_M \right\Vert  \leq  \frac{1}{M} \|\textbf{A}\|_{\max}\sum\limits_{i = 1}^M \Vert \textbf{b}_{i}\Vert_2$; 
\end{enumerate}
\end{minipage}%
\end{lemma}

\begin{proof}
\ref{itm:one}.
Let $\textbf{y} = \sum\limits_{i = 1}^M \textbf{a}_i$. The proof follows from an application of Jensen's inequality on the convex function $\phi(\textbf{y}) = \|\textbf{y}\|^2$. 
\\
\ref{itm:seven}. From the definition of Hadamard product, we have
\begin{align*}
\|\textbf{A}\circ \textbf{B}\|_F  = \sqrt{\sum_{i = 1}^H \sum_{j = 1}^M (a_{i,j} b_{i,j})^2} \leq \|\textbf{A}\|_{\max} \sqrt{\sum_{i = 1}^H \sum_{j = 1}^M b_{i,j}^2} & \leq  \|\textbf{A}\|_{\max} \sum_{i = 1}^H \sum_{j = 1}^M |b_{i,j}|
\\
&= \|\textbf{A}\|_{\max}\|\textbf{B}\|_{1,1} .     
\end{align*}
\ref{itm:eight}. It follows from the above inequality that    
\begin{align*}
\|\textbf{A}\circ \textbf{B}\|_F \leq \|\textbf{A}\|_{\max} \sqrt{\sum_{i = 1}^H \sum_{j = 1}^M b_{i,j}^2} & \leq   \|\textbf{A}\|_{\max}\|\textbf{B}\|_{F}.     
\end{align*}
\ref{itm:four}. Note that $\|\textbf{A}\otimes \textbf{B}\|_2  = \|\textbf{A}\|_2 \|\textbf{B}\|_2$. This, together with the fact that $\textbf{A}\circ \textbf{B}$ is a principle submatrix of  $\textbf{A}\otimes  \textbf{B}$ implies 
$$\|\textbf{A}\circ \textbf{B}\|_2    \leq \|\textbf{A}\otimes \textbf{B}\|_2 = \|\textbf{A}\|_2 \|\textbf{B}\|_2.$$
\\
\ref{itm:five}. Form the definition of Hadamard product, we have
\begin{align*}
\|\textbf{A} \circ \textbf{B}\|_F^2 = \sum_{i = 1}^H \sum_{j = 1}^M (a_{i,j}b_{i,j})^2 \leq  \left(\sum_{i = 1}^H \sum_{j = 1}^M a_{i,j}^2\right) \left(\sum_{i = 1}^H \sum_{j = 1}^M b_{i,j}^2\right) = \|\textbf{A}\|_F^2 \|\textnormal{B}\|_F^2,  
\end{align*}
where the inequality follows from Cauchy-Schwarz. 
\\
\ref{itm:two}. Observe that 
\begin{align*}
\left\Vert \frac{1}{M} (\textbf{A} \circ \textbf{B})\bm{1}_M \right\Vert^2 = \left\Vert \frac{1}{M} \sum\limits_{i = 1}^M\textbf{a}_i \circ \textbf{b}_i \right\Vert^2  \leq \frac{1}{M}  \sum\limits_{i = 1}^M\|\textbf{a}_i \circ \textbf{b}_i\|^2 & \leq \frac{1}{M} \|\textbf{A}\|^2_{\max}  \sum\limits_{i = 1}^M \|\textbf{b}_i\|^2,
\end{align*}
where the first and the second inequalities follow from~\ref{itm:one} and \ref{itm:eight}, respectively. \\
\ref{itm:three}. Similar to \ref{itm:two}, we obtain
\begin{align*}
\left\Vert \frac{1}{M} (\textbf{A} \circ \textbf{B})\bm{1}_M \right\Vert \leq \frac{1}{M}  \sum\limits_{i = 1}^M\|\textbf{a}_i \circ \textbf{b}_i\| \leq \frac{1}{M} \|\textbf{A}\|_{\max}  \sum\limits_{i = 1}^M \|\textbf{b}_i\|.
\end{align*}
\end{proof}

The following lemmas provide upper bounds on the norm of momentum vectors $\m_{i,k}$,$\sVtori_{i,k}$ and their product defined in Algorithm~\ref{alg:1}.

\begin{lemma}\label{lem:useful_bounds}[\citep{tran2019convergence}, Lemma 4.2] For each $i \in \{1,\cdots,M\}$ and $k \in \{1 \cdots N \}$, if $\|\seg_{i,k}\|_{\infty} \leq G_{\infty}$, then we have $\|\m_{i,k}\|_{\infty} \leq G_{\infty}$ and $\|\sVtori_{i,k}\|_{\infty} \leq G_{\infty}^2$. 
\end{lemma}

\begin{lemma}\label{VM_bound}
Assume $\gamma := \beta_{1,1}/\beta_2 \leq 1$ and let $\tilde{v}_{r,i,k}^{-\frac{1}{2}}$, and $m_{r,i,k}$ represent the values of the $r^{th}$ coordinate of vectors $\sVtori_{i,k}^{-\frac{1}{2}}$ and $\m_{i,k}$, respectively. Then, for each $i \in \{1,\cdots,M\}$, $k \in \{1 \cdots N \}$, and $r \in \{1, \cdots, d\}$, we have
\begin{align*}
|\tilde{v}_{r,i,k}^{-\frac{1}{2}} m_{r,i,k-1}| \leq \frac{1}{\sqrt{u_c}}, 
\end{align*}
\end{lemma}
where $u_c := (1-\beta_3)(1-\beta_{1,1}) (1-\beta_2) (1-\gamma). $
\begin{proof}
From the update rule of Algorithm~\ref{alg:1}, we have
$$\tilde{v}_{r,i,k} = \beta_3 \tilde{v}_{r,i,k-1} + (1-\beta_{3})\max(\tilde{v}_{r,i,k-1}, v_{r,i,k}),$$ 
which implies that $ \tilde{v}_{r,i,k} \geq (1-\beta_{3}) v_{r,i,k}$. 

It can easily be seen from the update rule of $\m_{i,k}$ and $ \wi_{i,k}$ in Algorithm~\ref{alg:1} that 
\begin{align*}
     m_{r,i,k} = \sum_{s = 1}^k \left( \prod\limits_{l = s+1}^k \beta_{1,l} \right)  (1-\beta_{1,s}) \hat{g}_{r,i,s}, \quad \text{and} \quad  v_{r,i,k}   = (1-\beta_{2})\sum_{s = 1}^k \beta_{2}^{k-s}\hat{g}_{r,i,s}^2. 
\end{align*}
Thus,
\begin{align}\label{eqn:bvb}
\nonumber
|v_{r,i,k}^{-\frac{1}{2}} m_{r,i,k-1}|^2 \leq |v_{r,i,k-1}^{-\frac{1}{2}} m_{r,i,k-1}|^2 &\leq    \frac{\left(\sum\limits_{s = 1}^{k-1} \left( \prod\limits_{l = s+1}^{k-1} \beta_{1,l} \right)  (1-\beta_{1,s}) \hat{g}_{r,i,s}\right)^2 }{(1-\beta_2)\sum\limits_{s = 1}^{k-1} \beta_{2}^{k-s-1} \hat{g}_{r,i,s}^2} 
\\
& \leq   \frac{\left(\sum\limits_{s = 1}^{k-1} \left( \prod\limits_{l = s+1}^{k-1} \beta_{1,l} \right)  \hat{g}_{r,i,s}\right)^2 }{(1-\beta_2)\sum\limits_{s = 1}^{k-1} \beta_{2}^{k-s-1} \hat{g}_{r,i,s}^2},
\end{align}
where the first inequality follows since ${v}_{r,i,k}^{-\frac{1}{2}} \leq {v}_{r,i,k-1}^{-\frac{1}{2}}$ for all $r \in [d]$ and the last inequality uses  our assumption that $\beta_{1,s} \leq 1$ for all $s \geq 1$.

Next, let $\pi_s =  \prod\limits_{l = s+1}^{k-1} \beta_{1,l}$. Since $\beta_{1,l}$ is decreasing, we get $\pi_s \leq \beta_{1,1}^{k-s-1}$. This, together with $(\sum_{i} a_i b_i)^2 \leq (\sum_{i} a_i^2) (\sum_{i} b_i^2)$ implies that  
\begin{align*}
\frac{\left(\sum\limits_{s = 1}^{k-1}  \pi_s  \hat{g}_{r,i,s}\right)^2 }{(1-\beta_2)\sum\limits_{s = 1}^{k-1} \beta_{2}^{k-s-1}\hat{g}_{r,i,s}^2} & \leq \frac{(\sum\limits_{s = 1}^{k-1} \pi_s)(\sum\limits_{s = 1}^{k-1} \pi_s  \hat{g}_{r,i,s}^2 )}{(1-\beta_2)\sum\limits_{s = 1}^{k-1} \beta_{2}^{k-s-1}\hat{g}_{r,i,s}^2}
\\
&\leq \frac{1}{1-\beta_2} (\sum_{s = 1}^{k-1}  \pi_s)\left(\sum_{s = 1}^{k-1} \frac{\pi_s  \hat{g}_{r,i,s}^2}{\beta_{2}^{k-s-1}\hat{g}_{r,i,s}^2}\right) \\
& \leq \frac{1}{1-\beta_2}  (\sum_{s = 1}^{k-1}  \pi_s)\sum_{s = 1}^{k-1} \frac{\pi_s}{\beta_{2}^{k-s-1}} 
\\
&\leq \frac{1}{1-\beta_2}\frac{1}{1-\beta_{1,1}} \frac{1}{1- \gamma}. 
\end{align*}
where the last inequality follows from the assumption that $\gamma= \frac{\beta_{1,1}}{\beta_2} \leq 1$.

Finally, substituting the above inequality into \eqref{eqn:bvb} yields the desired result.
\end{proof}

\begin{lemma}\label{lem:v_diff}
For any $ i \in [M]$, we have
\begin{enumerate}[label=(\roman*)]
\item \label{itm:one_Vp} $\sum\limits_{k=1}^{N} \|\sVtori_{i,k}^{p} - \sVtori_{i,k-1}^{p} \|_{1} \leq \sum_{r=1}^d \max \left( \tilde{v}_{r,i,0}^{p} , \tilde{v}_{r,i,N}^{p} \right)$; \quad  \text{and}
\item \label{itm:two_Vp} $ \sum\limits_{k=1}^{N} \|\sVtori_{i,k}^{p} - \sVtori_{i,k-1}^{p} \|_{1}^2 \leq \sum_{r=1}^d  \tilde{v}_{r,i,0}^{p} \max \left(\tilde{v}_{r,i,0}^{p} , \tilde{v}_{r,i,N}^{p} \right)$;
\end{enumerate}
where the vector powers are considered to be element-wise.
\end{lemma}
\begin{proof}
\ref{itm:one_Vp}~If $p>0$, we have from the update rule of $ \sVtori_{i,k}$ in Algorithm~\ref{alg:0} that 
\begin{align*}
\nonumber
\sum_{k=1}^{N}  \left\Vert \sVtori_{i,k}^{p} - \sVtori_{i,k-1}^{p}\right\Vert_1  =  \sum_{k=1}^{N}  \sum_{r=1}^d  (\tilde{v}_{r,i,k}^{p} - \tilde{v}_{r,i,k-1}^{p})  & = \sum_{r=1}^d \sum_{k=1}^{N}  (\tilde{v}_{r,i,k}^{p} - \tilde{v}_{r,i, k-1}^{p}) 
\\
& \leq   \sum_{r=1}^{d}  \tilde{v}_{r,i,N}^{p} , 
\end{align*}
where the first equality is due to the fact that each element of $\sVtori_k^p, p >0,$  is increasing in $k$ and the last inequality uses a telescopic sum. Next, we consider the case when $p<0$. It can easily be seen that 
\begin{align*}
\nonumber
 \sum_{k=1}^{N} \left\Vert \sVtori_{i,k}^{p} - \sVtori_{i,k-1}^{p}\right\Vert_1  &= \sum_{k=1}^{N}  \sum_{r=1}^d  (-\tilde{v}_{r,i,k}^{q} + \tilde{v}_{r,i,k-1}^{q}) 
 \leq   \sum_{r=1}^{d}  \tilde{v}_{r,i,0}^{p}.
\end{align*}
\\
\ref{itm:two_Vp}~For $p>0$, it follows that 
\begin{align*}
\nonumber
\sum_{k=1}^{N}  \left\Vert \sVtori_{i,k}^{p} - \sVtori_{i,k-1}^{p}\right\Vert_1^2 \leq 
 \sum_{k=1}^{N}   \sum_{r=1}^d \left(\tilde{v}_{r,i,k}^{p} - \tilde{v}_{r,i,k-1}^{p}\right) \tilde{v}_{r,i,k}^{p}  & \leq   \sum_{k=1}^{N}  \sum_{r=1}^d \left(\tilde{v}_{r,i,k}^{p} - \tilde{v}_{r,i,k-1}^{p}\right) \tilde{v}_{r,i,N}^{p} \\
  & \leq \sum_{r=1}^{d}  \left(\tilde{v}_{r,i,0}^{p} - \tilde{v}_{r,i,N}^{p}\right)  \tilde{v}_{r,i,N}^{p}  \\
& \leq \sum_{r=1}^{d}  \tilde{v}_{r,i,0}^{p} \tilde{v}_{r,i,N}^{p}.  
\end{align*}
Next, we consider the case of $p<0$. It can be seen that 
\begin{align*}
\nonumber 
\sum_{k=1}^{N}  \left\Vert \sVtori_{i,k}^{p} - \sVtori_{i,k-1}^{p}\right\Vert_1^2  \leq 
 \sum_{k=1}^{N}   \sum_{r=1}^d (-\tilde{v}_{r,i,k}^{p} + \tilde{v}_{r,i,k-1}^{p}) (\tilde{v}_{r,i,k-1}^{p}) & \leq   \sum_{k=1}^{N}  \sum_{r=1}^d (-\tilde{v}_{r,i,k}^{p} + \tilde{v}_{r,i,k-1}^{p}) \tilde{v}_{r,i,0}^{p} \\
& \leq    \sum_{r=1}^{d}  \tilde{v}_{r,i,0}^{p} \tilde{v}_{r,i,0}^{p}.  
\end{align*}
\end{proof}
\begin{lemma}\label{lem:netwrok_rho}[\citep{lian2017can}, Lemma 5] Under Assumption~\ref{assumption:graph}, we have
\begin{align*}
    \left\Vert \frac{1}{M} \bm{1}_M - \W^{k} \e_i\right\Vert \leq \rho^k, \quad \forall i \in\{1,\cdots,M\}.
\end{align*}
\end{lemma}

\begin{rmk}\label{rmk:index}
Assume $a_0 = 0$. We have the following useful change of indices,
\begin{align*}
\sum\limits_{k = 1}^N \sum\limits_{s = 1}^{k}a_{s-1}b^{k-s} = \sum\limits_{k = 1}^{N-1} a_k\sum\limits_{s = 0}^{N-k-1}b^{s}.
\end{align*}
\end{rmk}

\begin{lemma}\label{lem:xstar}
Suppose Assumptions~\ref{assumption:bounded-space} and~\ref{assumption:graph} hold. Let $G_{0}^2 \leq \|\sVtori_{i,0}\|_{\infty} \leq G_{\infty}^2, \;\;\forall i \in \{1,\cdots,M\}$ in Algorithm~\ref{alg:1}. Then we get
\begin{align*}
&\sum_{k = 1}^N \E\left\|\lpower_{k-1}^{\frac{1}{4}} \circ(\X_{k-1}\W^t-\X_*)\right\|^2_F - \E\left\|\lpower_{k-1}^{\frac{1}{4}}\circ(\X_{k}-\X_*)\right\|^2_F  \\
& \leq    \frac{\eta M\sqrt{du_c^{-1}}\rho^{2t} \beta_{1,1}}{D(1-\rho^t) (1-\kappa)} +  N\eta \frac{\rho^t}{1-\rho^t} \left[G_0 G_{\infty} M D^{-1}\sqrt{d}\right] +\sqrt{M}d.
\end{align*}
\end{lemma}
\begin{proof}
We note that  
\begin{align}\label{eq:upp0} 
\nonumber 
 &\left\|\lpower_{k-1}^{\frac{1}{4}} \circ(\X_{k-1}\W^t-\X_*)\right\|^2_F - \left\|\lpower_{k-1}^{\frac{1}{4}}\circ(\X_{k}-\X_*)\right\|^2_F 
 \\ \nonumber
 & =  \left(\left\|\lpower_{k-1}^{\frac{1}{4}} (\X_{k-1}\W^t-\X_*)\right\|_F - \left\|\lpower_{k-1}^{\frac{1}{4}}\circ(\X_{k}-\X_*)\right\|_F\right) 
 \\ \nonumber
 &\left( \left\|\lpower_{k-1}^{\frac{1}{4}}\circ(\X_{k-1}\W^t-\X_*)\right\|_F + \left\|\lpower_{k-1}^{\frac{1}{4}}\circ(\X_{k}-\X_*)\right\|_F\right)
 \\
 & \leq  \left(\left\|\lpower_{k-1}^{\frac{1}{4}}\circ (\X_{k-1}\W^t-\X_*)\right\|_F - \left\|\lpower_{k-1}^{\frac{1}{4}}\circ(\X_{k}-\X_*)\right\|_F\right) \left( 2\sqrt{G_{\infty}M} D\right).
\end{align}

On the other hand, 
\begin{align*}
& \left\|\lpower_{k-1}^{\frac{1}{4}} \circ (\X_{k-1}\W^t-\X_*)\right\|_F - \left\|\lpower_{k-1}^{\frac{1}{4}}(\X_{k}-\X_*)\right\|_F
\\
& = \|\lpower_{k-1}^{\frac{1}{4}} \circ(\X_{k-1}\W^t-\X_* - \bar\X_{k-1} + \bar\X_{k-1})\|_F - \left\|\lpower_{k-1}^{\frac{1}{4}}\circ (\X_{k}-\X_* -\bar\X_k + \bar\X_k)\right\|_F
\\
& \leq \|\lpower_{k-1}^{\frac{1}{4}} \circ(\X_{k-1}\W^t- \bar\X_{k-1}) \|_F + \|\lpower_{k-1}^{\frac{1}{4}} \circ(\bar\X_{k-1} - \X_* ) \|_F 
\\
& - \left\|\lpower_{k-1}^{\frac{1}{4}}\circ (\bar\X_{k}-\X_*)\right\|_F + \|\lpower_{k-1}^{\frac{1}{4}} \circ (\X_k - \bar\X_k)\|_F
\\
& \leq \|\lpower_{k-1}^{\frac{1}{4}} \circ(\X_{k-1}\W^t- \bar\X_{k-1}) \|_F + \|\lpower_{k-1}^{\frac{1}{4}} \circ(\bar\X_{k-1} - \X_* ) \|_F 
\\
& - \left\|\lpower_{k}^{\frac{1}{4}}\circ (\bar\X_{k}-\X_*)\right\|_F + \left\|(\lpower_{k-1}^{\frac{1}{4}} - \lpower_{k}^{\frac{1}{4}})\circ (\bar\X_{k}-\X_*)\right\|_F +  \|\lpower_{k}^{\frac{1}{4}} \circ (\X_k - \bar\X_k)\|_F.
\end{align*}

For the first term on R.H.S of above equation we have,
\begin{align*}
 & \left\|\lpower_{k-1}^{\frac{1}{4}} \circ(\X_{k-1}\W^t- \bar\X_{k-1}) \right\|_F \leq \sqrt{G_{\infty}} \sum_{i = 1}^M \left\| -\eta\sum_{s=1}^{k-1} \di_{s}\W^{t(k-s+1)} \e_{i}  +\frac{\eta}{M}\sum_{s=1}^{k-1} \di_{s}\bm 1_M \right\|_F
 \\
 & = \eta \sqrt{G_{\infty}} \sum_{i = 1}^M \underbrace{\left\| \sum_{s=1}^{k-1} \di_{s} (\W^{t(k-s+1)} \e_{i} - \frac{1}{M}\bm 1_M)\right\|}_{S_{1,k}}. 
\end{align*}

For the third term on R.H.S we have,
\begin{align*}
\left\|(\lpower_{k-1}^{\frac{1}{4}} - \lpower_{k}^{\frac{1}{4}})\circ (\bar\X_{k}-\X_*)\right\|_F \leq D \left\|\lpower_{k-1}^{\frac{1}{4}} - \lpower_{k}^{\frac{1}{4}}\right\|_F.
\end{align*}
Additionally, for the last term on the R.H.S we get
\begin{align*}
\|\lpower_{k}^{\frac{1}{4}} \circ (\X_k - \bar\X_k)\|_F &\leq \sqrt{G_{\infty}} \sum_{i = 1}^M  \left\|-\eta \sum_{s=1}^{k} \di_{s}\W^{t(k-s+1)}\e_i  +\frac{\eta}{M}\sum_{s=1}^{k} \di_{s} \bm 1_M \right\|_F
\\
& = \eta \sqrt{G_{\infty}} \sum_{i = 1}^M \underbrace{\left\| \sum_{s=1}^{k} \di_{s} (\W^{t(k-s+1)} \e_{i} - \frac{1}{M}\bm 1_M) \right\|}_{S_{2,k}}, 
\end{align*}
\textbf{Bounding $S_{1,k}$.} From the update rule of $\di_{s}$ defined in \eqref{eq:update_D}, we get
\begin{align}\label{eqn:s1-3}
\nonumber
S_{1,k} &= \left\|\sum_{s=1}^{k-1} \di_{s} \left(\frac{1}{M}\mathbf{1}_M-\W^{t(k-s+1)}\mathbf{e}_i\right)\right\|
\\
\nonumber
&{\leq} \sum_{s=1}^{k-1} \|\di_{s}\|_F \left\|\frac{1}{M}\mathbf{1}_M-\W^{t(k-s+1)}\mathbf{e}_i\right\|
\\
& \leq  \sum_{s=1}^{k-1} \left(  \beta_{1,{s}} \left\| \Vt_{s} \circ \M_{s-1} \right\|_F+ (1-\beta_{1,s})  \left\|\Vt_{s} \circ \eg_{s})\right\|_F \right)\rho^{t(k-s+1)},
\end{align}
where the first inequality uses  Lemma~\ref{lem:hadamard}\ref{itm:five} and the second inequality follows from Lemma~\ref{lem:netwrok_rho}.

From Lemma~\ref{VM_bound}, we have  
\begin{align}\label{eq:mome1}
\left\|\Vt_{s} \circ \M_{s-1}\right\|_F \leq \left\|\Vt_{s-1} \circ \M_{s-1}\right\|_F  \leq \sqrt{Mdu_c^{-1}}.
\end{align}
Further, by using Lemma~\ref{lem:hadamard}\ref{itm:five}, we get 
\begin{align}\label{eq:mome2}
\| \Vt_{s}  \circ \eg_{s}\|_F  &\leq \left(\| \Vt_{s}\|_{\max}\| \eg_{s}\|_F\right)  \leq   \| \Vt_{0}\|_{\max} \|\eg_{s}\|_F \leq  G_{0}  \sqrt{Md}G_{\infty}, 
\end{align}
where the first inequality follows from Lemma~\ref{lem:hadamard}\ref{itm:five}; the second inequality uses $ \|\Vt_{s-1}\|_F \leq \|\Vt_{0}\|_F$; and the last inequality uses 
Assumption~\ref{assumption:g_bounded}.

We substitute~\eqref{eq:mome1} and~\eqref{eq:mome2} into \eqref{eqn:s1-3} to get 
\begin{align}\label{eq:s11k1}
\nonumber
S_{1,k} &\leq  \sqrt{Mdu_c^{-1}}  \sum_{s=1}^{k-1} \beta_{1,{s}}  \rho^{t(k-s+1)} + G_{0}  \sqrt{Md}G_{\infty} \sum_{s=1}^{k-1}  \rho^{t(k-s+1)} 
\\
&\leq  \sqrt{Mdu_c^{-1}} \sum_{s=1}^{k-1} \beta_{1,{s}}  \rho^{t(k-s+1)} + G_{0}  \sqrt{Md}G_{\infty} \frac{\rho^t}{1-\rho^t}.
\end{align}
A similar approach yileds
\begin{align}\label{eq:s11k2}
\nonumber
S_{2,k} &\leq  \sqrt{Mdu_c^{-1}}  \sum_{s=1}^{k} \beta_{1,{s}}  \rho^{t(k-s+1)} + G_{0}  \sqrt{Md}G_{\infty} \sum_{s=1}^{k}  \rho^{t(k-s+1)} 
\\
&\leq  \sqrt{Mdu_c^{-1}} \sum_{s=1}^{k} \beta_{1,{s}}  \rho^{t(k-s+1)} + G_{0}  \sqrt{Md}G_{\infty} \frac{\rho^t}{1-\rho^t}.
\end{align}

Thus, summing over $k$ and taking the expectation we obtain
\begin{align*}
\left( 2\sqrt{G_{\infty}M} D\right) ^{-1}&\sum_{k = 1}^N \E\left\|\lpower_{k-1}^{\frac{1}{4}} \circ(\X_{k-1}\W^t-\X_*)\right\|^2_F - \E\left\|\lpower_{k-1}^{\frac{1}{4}}\circ(\X_{k}-\X_*)\right\|^2_F     
\\
& \leq   \left(2\eta \sqrt{G_{\infty}}M \left[\sqrt{Mdu_c^{-1}}\frac{\rho^{2t} \beta_{1,1}}{(1-\rho^t) (1-\kappa)} + G_{0}  \sqrt{Md}G_{\infty} N\frac{\rho^t}{1-\rho^t} \right]\right) 
\\
& + D\sqrt{MG_{\infty}}+ DMd\sqrt{G_{\infty}}.
\end{align*}
As a result,
\begin{align*}
&\sum_{k = 1}^N \left(\E\left\|\lpower_{k-1}^{\frac{1}{4}} \circ(\X_{k-1}\W^t-\X_*)\right\|^2_F - \E\left\|\lpower_{k-1}^{\frac{1}{4}}\circ(\X_{k}-\X_*)\right\|^2_F\right)     
\\
& \leq   \frac{\eta M\sqrt{du_c^{-1}}\rho^{2t} \beta_{1,1}}{D(1-\rho^t) (1-\kappa)} +  N\eta \frac{\rho^t}{1-\rho^t} \left[G_0 G_{\infty} M D^{-1}\sqrt{d}\right] +\sqrt{M}d.
\end{align*}
This completes the proof.
\end{proof}

\begin{lemma}\label{lem:network_power_2}
Suppose $ \beta_{1,k} = \beta_1 \kappa^{k-1} $ for some $\kappa \in (0,1)$.  
Then, for the sequence $\{\z_{i,k}\}_{i,k=1}^{M,N}$ generated by Algorithm~\ref{alg:1} we have
\begin{align*}
\frac{1}{NM}\sum_{k=1}^{N}\sum_{i=1}^M \E \left\|\bar\z_k-\z_{i,k}\right\|^2 \leq \frac{2\eta^2 Mdu_c^{-1} \beta_{1,1}^2 \rho^{2t}}{N(1-\kappa^2)(1-\rho^{2t})} + \frac{2\eta^2 Md \rho^{2t}}{(1-\beta_2)(1-\rho^{2t})}.
\end{align*}
\end{lemma}
\begin{proof}
It follows from \eqref{eq:updatez} that 
\begin{align}\label{z_power_1}
\nonumber &\frac{1}{M}\sum_{i=1}^{M}\left\|\bar\z_k-\z_{i,k}\right\|^2
\\ \nonumber
&= \frac{1}{M}\sum_{i=1}^{M}\left\|\eta \sum_{s=1}^{k} \di_{s-1}\left(\frac{1}{M}\mathbf{1}_M-\W^{t(k-s+2)}\mathbf{e}_i\right) + \eta  \di_{k-1}\left(\frac{1}{M}\mathbf{1}_M-\W^t\mathbf{e}_i\right)\right\|^2 
\\ \nonumber
& \leq \frac{2\eta^2}{M}\sum_{i=1}^{M} \underbrace{\left\|\sum_{s=1}^{k} \di_{s-1}\left(\frac{1}{M}\mathbf{1}_M-\W^{t(k-s+2)}\mathbf{e}_i\right)\right\|^2}_{S_{1,k}}\\
&\quad + \frac{2\eta^2}{M} \sum_{i=1}^{M} \underbrace{\left\| \di_{k-1}\left(\frac{1}{M}\mathbf{1}_M-\W^t\mathbf{e}_i\right)\right\|^2}_{S_{2,k}},
\end{align}
where the last inequality is obtained from the triangular inequality. 

In the following, we provide upper bounds for $S_{1,k}$ and $S_{2,k}$.
\vspace{.5cm}
\\
\textbf{Bounding $S_{1,k}$.} From the update rule of $\di_{s-1}$ defined in \eqref{eq:update_D}, we get
\begin{align}\label{eqn:s1-31}
\nonumber
S_{1,k} &= \left\|\sum_{s=1}^{k} \di_{s-1} \left(\frac{1}{M}\mathbf{1}_M-\W^{t(k-s+2)}\mathbf{e}_i\right)\right\|^2
\\
\nonumber
&{\leq} \sum_{s=1}^{k}\|\di_{s-1}\|_F^2 \left\|\frac{1}{M}\mathbf{1}_M-\W^{t(k-s+2)}\mathbf{e}_i\right\|^2
\\
& \leq 2 \sum_{s=1}^{k} \left(  \beta_{1,{s-1}}^2  \left\| \Vt_{s-1} \circ \M_{s-2} \right\|_F^2+ (1-\beta_{1,s-1})^2  \left\|\Vt_{s-1} \circ \eg_{s-1}\right\|_F^2 \right)\rho^{2t(k-s+2)},
\end{align}
where the first inequality uses  Lemma~\ref{lem:hadamard}\ref{itm:five} and the second inequality follows from Lemma~\ref{lem:netwrok_rho}.

From Lemma~\ref{VM_bound}, we have  
\begin{align}\label{eq:mome11}
\left\|\Vt_{s-1} \circ \M_{s-2}\right\|_F \leq \left\|\Vt_{s-2} \circ \M_{s-2}\right\|_F  \leq \sqrt{Mdu_c^{-1}}.
\end{align}
Further, by following similar steps as in Lemma~\ref{VM_bound} we get 
\begin{align}\label{eq:mome21}
\| \Vt_{s-1}  \circ \eg_{s-1}\|_F  &\leq \sqrt{\frac{Md}{1-\beta_2}}.
\end{align}

We substitute~\eqref{eq:mome11} and~\eqref{eq:mome21} into \eqref{eqn:s1-31} to get 
\begin{align}\label{eq:s11k1}
\nonumber
S_{1,k} &\leq  2Mdu_c^{-1}  \sum_{s=1}^{k} \beta_{1,{s-1}}^2  \rho^{2t(k-s+2)} + \frac{2Md}{1-\beta_2} \sum_{s=1}^{k}  \rho^{2t(k-s+2)} 
\\
&\leq  2Mdu_c^{-1} \sum_{s=1}^{k} \beta_{1,{s-1}}^2  \rho^{2t(k-s+2)} + \frac{2Md}{1-\beta_2} \frac{\rho^{4t}}{1-\rho^{2t}} .
\end{align}
\vspace{.5cm}
\\
\textbf{Bounding $S_{2,k}$.} 
It follows from the update rule of $\di_{s-1}$ in \eqref{eq:update_D} that 
\begin{align}\label{eq:S_2_k}
\nonumber
S_{2,k} &= \left\|\di_{k-1}\left(\frac{1}{M}\mathbf{1}_M-\W^t\mathbf{e}_i\right)\right\|^2
\\ \nonumber
& \leq  \left\| \left(\beta_{1,k-1} \Vt_{k-1} \circ \M_{k-2} + (1-\beta_{1,k-1})\Vt_{k-1} \circ \eg_{k-1}\right) \left(\frac{1}{M}\mathbf{1}_M-\W^t\mathbf{e}_i\right) \right\|^2 \\
\nonumber 
& \leq  2\beta_{1,{k-1}}^2  \left\|\Vt_{k-1} \circ \M_{k-2} \left(\frac{1}{M}\mathbf{1}_M-\W^t\mathbf{e}_i\right)\right\|^2 
\\ \nonumber
&+ (1-\beta_{1,k-1})^2  \left\|\Vt_{k-1} \circ \eg_{k-1} \left(\frac{1}{M}\mathbf{1}_M-\W^t\mathbf{e}_i\right) \right\|^2 \\
\nonumber 
& \leq  2\beta_{1,{k-1}}^2  \left\| \Vt_{k-1} \circ \M_{k-2} \right\|_F^2 \left\| \frac{1}{M}\mathbf{1}_M-\W^t\mathbf{e}_i \right\|^2 +   2\left\|\Vt_{k-1} \circ \eg_{k-1} \right\|_F^2 \left\| \frac{1}{M}\mathbf{1}_M-\W^t\mathbf{e}_i \right\|^2
\\
& \leq 2\beta_{1,k-1}^2  Mdu_c^{-1}\rho^{2t} +   \frac{2Md}{1-\beta_2}\rho^{2t}. 
\end{align}
\vspace{.5cm}

\textbf{Combining intermediate results.}  By substituting \eqref{eq:s11k1} and \eqref{eq:S_2_k} into \eqref{z_power_1} and summing over $k=1, \cdots, N $, we obtain
\begin{align}\label{eq:lem8:zb}
\nonumber
& \quad \frac{1}{NM}\sum_{k=1}^{N}\sum_{i=1}^M \E \left\|\bar\z_k-\z_{i,k}\right\|^2 
\\ \nonumber
&\leq \frac{1}{N} \sum_{k = 1}^N \frac{\eta^2}{M} \sum_{i = 1}^M \left( 2Mdu_c^{-1} \sum_{s=1}^{k} \beta_{1,{s-1}}^2  \rho^{2t(k-s+2)} + \frac{2Md}{1-\beta_2} \frac{\rho^{4t}}{1-\rho^{2t}}\right)
\\ \nonumber
& + \frac{1}{N} \sum_{k = 1}^N \frac{\eta^2}{M} \sum_{i = 1}^M \left(2\beta_{1,k-1}^2  Mdu_c^{-1}\rho^{2t} +   \frac{2Md}{1-\beta_2}\rho^{2t} \right) 
\\ \nonumber
& \leq \frac{2\eta^2 Mdu_c^{-1} \beta_{1,1}^2 \rho^{4t}}{N(1-\kappa^2)(1-\rho^{2t})} + \frac{2\eta^2 Md\rho^{4t}}{(1-\beta_2)(1-\rho^{2t})} + \frac{2\eta^2 Mdu_c^{-1} \rho^{2t} \beta_{1,1}^2}{N(1-\kappa^2)} +  \frac{2\eta^2 Md \rho^{2t}}{1-\beta_2} 
\\
&\leq \frac{2\eta^2 Mdu_c^{-1} \beta_{1,1}^2 \rho^{2t}}{N(1-\kappa^2)(1-\rho^{2t})} + \frac{2\eta^2 Md \rho^{2t}}{(1-\beta_2)(1-\rho^{2t})},
\end{align}
where the second inequality uses the assumption that $ \beta_{1,k} = \beta_1 \kappa^{k-1} $ for some $\kappa \in (0,1)$ and follows by an application of  Remark~\ref{rmk:index}, wherein we obtain 
$$
\sum_{k=1}^N \sum_{s=1}^{k} \beta_{1,{s-1}}^2 \rho^{2t(k-s+2)}= \sum_{k=1}^{N-1} \beta_{1,{k}}^2  \sum_{s=0}^{N-k-1} \rho^{2t(s+2)} \leq \frac{\beta_{1,1}^2}{1-\kappa^2} \frac{\rho^{4t}}{1-\rho^{2t}}. 
$$

\end{proof}
With the lemmas provided above, we prove the main theorem in the next section. 
\subsection{Proof of Theorem~\ref{thm:revised_main}}

\begin{proof}
We observe that
\begin{align*}
\eta (1-\beta_{1,k})\frac{1}{M}\left(\Vt_{k} \circ \bg_k\right)\mathbf{1}_M&=\bar\z_k-\bar\x_k+\frac{1}{M}\left(\X_{k-1}\W^t-\eta \di_k \right)\mathbf{1}_M\\
& -\left(\bar\z_k-\eta (1-\beta_{1,k}) \frac{1}{M}\left(\Vt_{k} \circ \bg_k\right)\mathbf{1}_M\right). 
\end{align*}
This, together with Lemma~\ref{lem:hadamard}\ref{itm:one} gives 
\begin{align}\label{eq:r12}
\nonumber \eta^2 (1-\beta_{1,k})^2 \left\|\frac{1}{M}\left(\Vt_{k} \circ \bg_k\right)\mathbf{1}_M\right\|^2 &\leq 3\eta^2\underbrace{\left\|\frac{1}{M} \left(-\di_k+(1-\beta_{1,k}) \Vt_{k} \circ \bg_k\right)\bm{1}_M\right\|^2}_{R_{1,k}} \\
&+ 3 \underbrace{ \left(\left\|\bar{\z}_k-\bar{\x}_k\right\|^2+\left\|\bar{\z}_k-\bar{\x}_{k-1}\right\|^2\right)}_{R_{2,k}}.
\end{align}
From Lemma~\ref{lem:useful_bounds}, we have $\|\sVtori_{i,k}^{-\frac{1}{2}}\|_{\infty} \geq G_{\infty}^{-1}$ which implies that
\begin{align*}
\left\|\frac{1}{M}\left(\Vt_{k} \circ \bg_k\right)\bm{1}_M\right\|^2 = \left\| \frac{1}{M} \sum_{i = 1}^M \sVtori_{i,k}^{-\frac{1}{2}} \circ \sbg_k\right\|^2  & = \left\|  \sbg_k \circ \frac{1}{M}\sum_{i = 1}^M  \sVtori_{i,k}^{-\frac{1}{2}} \right\|^2\\
&\geq G_{\infty}^{-2} \|\sbg_k\|^2.  
\end{align*}
Thus, 
 \begin{align}\label{eq:thm_main}
 \nonumber
 \eta^2 (1-\beta_{1,1})^2 G_{\infty}^{-2} \|\sbg_k\|^2 & \leq  \eta^2 (1-\beta_{1,k})^2 \left\|\frac{1}{M}\left(\Vt_{k} \circ \bg_k\right)\mathbf{1}_M\right\|^2 \\
 &\leq 3 \eta^2 R_{1,k}  + 3R_{2,k}. 
 \end{align}
Next, we provide upper bounds for $R_{1,k}$ and $R_{2,k}$ defined in \eqref{eq:r12}.

\textbf{Bounding $R_{1,k}$.} From the definition of $\di_{k}$ in~\eqref{eq:update_D}, we have $$
-\di_k+(1-\beta_{1,k}) \Vt_{k} \circ \bg_k =-\beta_{1,k} \left(\Vt_{k} \circ \M_{k-1}\right) + (1-\beta_{1,k})\Vt_{k} \circ \left(\bg_k-\eg_k\right),$$
which implies that 
\begin{align}\label{R_1_k}
\nonumber
R_{1,k} & = \left\|\frac{1}{M} \left(-\beta_{1,k} \Vt_{k} \circ \M_{k-1} + (1-\beta_{1,k}) \Vt_{k}\circ \left( \bg_k-(\rg_k + \beps_k) \right) \right)\bm{1}_M\right\|^2
\\\nonumber
&\leq 3 \beta^2_{1,k} \left\|\frac{1}{M} \Vt_{k} \circ \M_{k-1}\bm{1}_M\right\|^2+3 (1-\beta_{1,k})^2\left\|\frac{1}{M} \Vt_k \circ \beps_k \bm{1}_M\right\|^2
\\
& \quad + 3(1-\beta_{1,k})^2 \left\Vert\frac{1}{M}\Vt_k \circ (\rg_k-\bg_k) \bm{1}_M\right\Vert^2.
\end{align}
The first equality is due to $\beps_k = \eg_k - \rg_k$, while the first inequality follows from Lemma~\ref{lem:hadamard}(i). 

For the first term on the R.H.S. of \eqref{R_1_k}, we have 
\begin{align}\label{rd11k}
&\left\|\frac{1}{M} \Vt_{k} \circ \M_{k-1}\bm{1}_M\right\|^2 = \left\|\frac{1}{M} \sum\limits_{l = 1}^M\sVtori_{l,k-1}^{-\frac{1}{2}} \circ \m_{l,k-1}\right\|^2 \leq \frac{d}{u_c},
\end{align}
where the inequality uses Lemma~\ref{VM_bound}. 

For the second term on the R.H.S. of \eqref{R_1_k}, we have
\begin{align}\label{rd12k}
&\left\|\frac{1}{M} \Vt_k \circ \beps_k \bm{1}_M\right\|^2 \leq    \|\Vt_k\|_{\max}^2  \frac{1}{M} \sum_{i = 1}^M \|\beps_{i,k}\|^2 \leq  \frac{1}{M G_0^2}  \sum\limits_{i = 1}^M \|\beps_{i,k}\|^2,
\end{align}
where the inequality uses Lemma~\ref{lem:hadamard}~\ref{itm:two} and our assumption that $\|\tilde{\wi}_{0,i}^{-\frac{1}{2}}\|_{\infty} \leq  1/G_{0}$.

For the third term on the R.H.S. of \eqref{R_1_k}, we have
\begin{align}\label{rd13k}
\nonumber
\left\Vert\frac{1}{M}\Vt_k \circ (\rg_k-\bg_k) \bm{1}_M\right\Vert^2 &\leq  \|\Vt_k\|_{\max}^2  \frac{1}{M} \sum_{i = 1}^M \|\srg_{i,k} -\sbg_{i,k}\|^2 \\
&\leq  \frac{L^2}{M  G_0^2} \sum\limits_{i = 1}^M  \|\z_{i,k} - \bar{\z}_{k}\|^2,
\end{align}
where the last inequality uses Assumption~\ref{assumption:function}(\ref{assu:Lipshitz_gradient}). 

Substituting \eqref{rd11k}--\eqref{rd13k} into \eqref{R_1_k}, we obtain 
\begin{align}\label{R1k}
R_{1,k}  \leq \frac{3}{M}\sum \limits_{i = 1}^M \left(\frac{d\beta_{1,k}^2}{u_c} +  \frac{1}{G_0^2}  \|\beps_{i,k}\|^2 +  \frac{L^2}{G_0^2}\|\z_{i,k} - \bar{\z}_{k}\|^2\right).
\end{align}
\textbf{Bounding $R_{2,k}$.} 
% We note that $\bar{\x}_{k} = \frac{1}{M}\X_{k}\W \bm{1}_M =  \frac{1}{M}\X_{k}\bm{1}_M $ duo to Assumption~\ref{assumption:graph}. This, together with the update rule of $\X_{k}$ defined in ~\eqref{eq:update_zx} implies that
Let $\bPsi, \X \in \mathbb{R}^{d \times M}$ be two arbitrary matrices. From the update rule of $\X_{k}$ in Algorithm~\ref{alg:1}, we have  
\begin{align*}
&\nonumber  \left\|\bPsi\circ (\X_k-\X\right)\|_F^2 = \left\Vert \bPsi\circ (\X_{k-1}\W^t - \eta \di_k \W^t - \X)\right\Vert^2_F
\\
& = \left\Vert \bPsi\circ \left( \X_{k-1}\W^t- \X\right) - \eta \bPsi\circ \di_k\W^t \right\Vert_F^2 - \left\Vert \bPsi\circ \left(\X_{k-1}\W^t - \X_{k}\right) - \eta \bPsi\circ \di_k \W^t  \right\Vert_F^2
\\
& = \left\|\bPsi\circ (\X_{k-1}\W^t-\X)\right\|_F^2-\left\|\bPsi\circ (\X_{k-1}\W^t-\X_k)\right\|_F^2 -2\left\langle \bPsi\circ (\X_{k-1}\W^t-\X), \eta  \bPsi\circ \di_k\W^t \right\rangle_F \\
& + 2\left\langle \bPsi\circ (\X_{k-1}\W^t - \X_k), \eta  \bPsi\circ \di_k \W^t\right\rangle_F-2\left\langle \bPsi\circ \Z_k, \eta  \bPsi\circ \di_k\W^t\right\rangle_F + 2\left\langle \bPsi\circ \Z_k, \eta  \bPsi\circ \di_k\W^t\right\rangle_F
\\
& = \left\|\bPsi\circ (\X_{k-1}\W^t-\X)\right\|_F^2-\left\|\bPsi\circ (\X_{k-1}\W^t -\Z_k+ \Z_k-\X_k)\right\|_F^2 \\
& -2\left\langle \bPsi\circ (\Z_k-\X), \eta  \bPsi\circ \di_k \W^t\right\rangle_F + 2\left\langle  \bPsi\circ (\Z_k- \X_k), \eta  \bPsi\circ \di_k \W\right\rangle_F
\\
& = \left\|\bPsi\circ (\X_{k-1}\W^t-\X)\right\|_F^2-\left\|\bPsi\circ (\X_{k-1}\W^t-\Z_k)\right\|_F^2-\left\|\bPsi\circ (\Z_k-\X_k)\right\|_F^2 \\
& + 2\left\langle \bPsi\circ (\X-\Z_k), \eta  \bPsi\circ \di_k\W^t \right\rangle_F  + 2\left\langle \bPsi\circ (\X_k-\Z_k),\bPsi\circ (\X_{k-1}\W^t -\Z_k)\right\rangle_F 
\\ & +  2\left\langle \bPsi\circ (\X_k-\Z_k),-\eta \bPsi\circ \di_k\W^t \right\rangle_F,
\end{align*}
where the second equality follows since $\X_{k-1}\W^t - \X_{k} - \eta \di_k\W^t=0$.
 
Now, by substituting $\X =\X_* = [\x_*,\cdots,\x_*]^\top$ and $\bPsi = \lpower_{k-1}^{\frac{1}{4}}$ into the above equality and rearranging the terms we get  
\begin{align} \label{eq:R2k0}
\nonumber & \left\|\lpower_{k-1}^{\frac{1}{4}} \circ (\Z_k-\X_k)\right\|_F^2 + \left\|\lpower_{k-1}^{\frac{1}{4}} \circ (\X_{k-1}\W^t-\Z_k)\right\|_F^2 =   
\\ \nonumber
 &\underbrace{ \left\|\lpower_{k-1}^{\frac{1}{4}} \circ (\X_{k-1}\W^t-\X_*)\right\|_F^2 - \left\|\lpower_{k-1}^{\frac{1}{4}} \circ (\X_{k}-\X_*)\right\|_F^2}_{R_{2,0,k}}
+2 \eta \underbrace{\left\langle \lpower_{k-1}^{\frac{1}{4}} \circ (\X_*-\Z_k), \lpower_{k-1}^{\frac{1}{4}} \circ \di_k\W^t \right\rangle_F}_{R_{2,1,k}}\\
&+2 \eta \underbrace{\left\langle \lpower_{k-1}^{\frac{1}{4}} \circ (\X_k-\Z_k),  \lpower_{k-1}^{\frac{1}{4}} \circ  (\di_{k-1}\W^t - \di_k\W^t)\;\right\rangle_F}_{R_{2,2,k}}.   
\end{align}
Since the matrix $\W^t$ is doubly stochastic, we get $\sum_{i=1}^M [\W^t]_{i,j}=1$. Thus,
\begin{subequations}
\begin{align}\label{eq:low1}
\nonumber
    \left\|\bar \x_{k-1} - \bar \z_{k} \right\|^2 & =   \left\|\frac{1}{M} \sum_{j=1}^M\x_{j,k-1}-\frac{1}{M}\sum_{i=1}^M\z_{i,k}\right\|^2 \\
    \nonumber
    & =   \left\|\frac{1}{M}\sum_{j=1}^M \sum_{i=1}^M [\W^t]_{j,i}\x_{j,k-1}- \frac{1}{M}\sum_{i=1}^M  \z_{i,k}\right\|^2 \\
    \nonumber
        & =   \left\|\frac{1}{M}\sum_{i=1}^{M} \sum_{j=1}^M [\W^t]_{j,i}\x_{j,k-1}- \frac{1}{M}\sum_{i=1}^M\z_{i,k}\right\|^2 \\
        & \leq  \frac{1}{M} \sum_{i=1}^M \left\|\sum_{j=1}^{M}[\W^t]_{j,i}\x_{j,k-1}-\z_{i,k}\right\|^2 = \frac{1}{M} \|\X_{k-1}\W^t - \Z_{k}\|_F^2,
\end{align}
where the last inequality uses Lemma~\ref{lem:hadamard}\ref{itm:one}. Similarly, 
\begin{align}\label{eq:low2}
    \left\|\bar \z_{k} - \bar \x_{k} \right\|^2 & =   \left\|\frac{1}{M} \sum_{i=1}^M\z_{i,k}-\frac{1}{M}\sum_{i=1}^M\x_{i,k}\right\|^2 \leq  \frac{1}{M} \sum_{i=1}^{M}\left\|\z_{i,k}-\x_{i,k}\right\|^2 = \frac{1}{M} \|\Z_{k}-\X_{k}\|_F^2.
\end{align}
Additionally, we have
\begin{align}\label{eq:add}
\nonumber G_0 \|\Z_k - \X_k\|_F^2 &\leq \left\|\lpower_{k-1}^{\frac{1}{4}} \circ (\Z_k-\X_k)\right\|_F^2 
\\
 G_0 \|\X_{k-1}\W^t-\Z_k\|_F^2 &\leq \left\|\lpower_{k-1}^{\frac{1}{4}} \circ (\X_{k-1}\W^t-\Z_k)\right\|_F^2.  
\end{align}
\end{subequations} 

% \begin{align}\label{eq:low1}
% \nonumber
%     \left\|\bar \x_{k-1} - \bar \z_{k} \right\|^2 & =   \left\|\frac{1}{M} \sum_{j=1}^M\x_{j,k-1}-\frac{1}{M}\sum_{i=1}^M\z_{i,k}\right\|^2 \\
%     \nonumber
%     & =   \left\|\frac{1}{M}\sum_{j=1}^M \sum_{i=1}^M w_{j,i}\x_{j,k-1}- \frac{1}{M}\sum_{i=1}^M  \z_{i,k}\right\|^2 \\
%     \nonumber
%         & =   \left\|\frac{1}{M}\sum_{i=1}^{M} \sum_{j=1}^M w_{j,i}\x_{j,k-1}- \frac{1}{M}\sum_{i=1}^M\z_{i,k}\right\|^2 \\
%         & \leq  \frac{1}{M} \sum_{i=1}^M \left\|\sum_{j=1}^{M}w_{j,i}\x_{j,k-1}-\z_{i,k}\right\|^2,
% \end{align}
% where the last inequality uses Lemma~\ref{lem:hadamard}\ref{itm:one} and the fact that $\W$ is symmetric. Similarly, 
% \begin{align}\label{eq:low2}
%     \left\|\bar \z_{k} - \bar \x_{k} \right\|^2 & =   \left\|\frac{1}{M} \sum_{i=1}^M\z_{i,k}-\frac{1}{M}\sum_{i=1}^M\x_{i,k}\right\|^2 \leq  \frac{1}{M} \sum_{i=1}^{M}\left\|\z_{i,k}-\x_{k,i}\right\|^2.
% \end{align}
% Additionally we have,
% \begin{align}\label{eq:add}
% \nonumber G_0 \|\Z_k - \X_k\|_F^2 &\leq \E\left\|\lpower_{k-1}^{\frac{1}{4}} \circ (\Z_k-\X_k)\right\|_F^2 
% \\
%  G_0 \|\X_{k-1}\W-\Z_k\|_F^2 &\leq \E\left\|\lpower_{k-1}^{\frac{1}{4}} \circ (\X_{k-1}\W-\Z_k)\right\|_F^2.  
% \end{align}
We substitute the lower bounds in~\eqref{eq:low1}--~\eqref{eq:add} into \eqref{eq:R2k0} to get
\begin{align} \label{eq:R2k}
\left\|\bar \z_{k} - \bar \x_{k} \right\|^2 +  \left\|\bar \x_{k-1} - \bar \z_{k} \right\|^2  & \leq  \frac{R_{2,0,k}}{MG_0} + \frac{2 \eta}{MG_0}\left(R_{2,1,k} + R_{2,2,k}\right).   
\end{align}
Next, we provide upper bounds for the terms $R_{2,1,k}$ and $R_{2,2,k}$.
\vspace{0.5cm}
\\
\textbf{Bounding $R_{2,1,k}$.} It follows from the update rule of $\di_k$ in~\eqref{eq:update_D} that 
\begin{align}\label{eq:ddd}
\nonumber
\di_k &= \di_k - (1-\beta_{1,k})\Vt_{k-1}\circ \eg_{k}  +(1-\beta_{1,k})\Vt_{k-1}\circ \eg_{k}
\\
\nonumber
& = \beta_{1,k}\Vt_k\circ \M_{k-1} + (1-\beta_{1,k}) (\Vt_k-\Vt_{k-1})\circ \eg_{k} \\
&+ (1-\beta_{1,k})\Vt_{k-1}\circ \rg_k + (1-\beta_{1,k}) \Vt_{k-1}\circ (\eg_{k}- \rg_{k}). 
\end{align}
In the following, we multiply each term in \eqref{eq:ddd} by $\lpower_{k-1}^{\frac{1}{4}} \circ (\X_*-\Z_k)$ and then provide an upper bound for the resulting term.
\begin{subequations}

Using Lemmas~\ref{VM_bound},~\ref{lem:useful_bounds} and Assumption~\ref{assumption:bounded-space},  we get 
\begin{align}\label{eq:bbb1}
\left\langle \lpower_{k-1}^{\frac{1}{4}}\circ (\X_*-\Z_k), \lpower_{k-1}^{\frac{1}{4}}\circ \Vt_k\circ \M_{k-1} \right\rangle_F  & \leq  \sqrt{M} D G_{\infty}  \left\|\Vt_k\circ \M_{k-1}\right\|_F \leq MD G_{\infty} \sqrt{du_c^{-1}}.
\end{align}
Besides,
\begin{align}\label{eq:bbb2}
\nonumber
\left\langle \lpower_{k-1}^{\frac{1}{4}}\circ (\X_*-\Z_k),\lpower_{k-1}^{\frac{1}{4}}\circ (\Vt_k-\Vt_{k-1})\circ \eg_k \right\rangle_F & \leq  \sqrt{M}D G_{\infty}\| (\Vt_k-\Vt_{k-1})\circ \eg_k \|_F \\
\nonumber
&\leq   \sqrt{M}D G_{\infty} \|\eg_k\|_{\max}  \|\Vt_k-\Vt_{k-1} \|_{1,1} \\
&\leq   \sqrt{M}D G_{\infty}^2  \|\Vt_k-\Vt_{k-1}\|_{1,1}, 
\end{align}
where the second inequality is obtained from Lemma~\ref{lem:hadamard}\ref{itm:seven} and last inequality is due to Assumption~\ref{assumption:g_bounded}. 

In addition, from Assumption~\ref{assumption:bounded-spaceI}, we have 
\begin{align}\label{eq:bbb3}
\left\langle \lpower_{k-1}^{\frac{1}{4}}\circ (\X_*-\Z_k), \lpower_{k-1}^{-\frac{1}{4}}\circ \rg_k \right\rangle_F & = \sum\limits_{i = 1}^M \left\langle\x_*-\z_{i,k},\srg_{i,k} \right\rangle  \leq 0. 
\end{align} 

Also,
\begin{align}\label{eq:bbb4}
\left\langle \lpower_{k-1}^{\frac{1}{4}} \circ(\X_*-\Z_k), \lpower_{k-1}^{-\frac{1}{4}}\circ (\eg_{k}- \rg_{k})\right\rangle_F &= \left\langle\X_*-\Z_k,  \eg_{k}- \rg_{k}\right\rangle_F :=\boldsymbol{\Theta}_k.
\end{align}
\end{subequations}

Combining \eqref{eq:bbb1}--\eqref{eq:bbb4} yields
\begin{align}\label{eq:N_update}
 R_{2,1,k} & \leq  \beta_{1,k}  MD G_{\infty}\sqrt{du_c^{-1}}   + \sqrt{M}D G_{\infty}^2  \|\Vt_k-\Vt_{k-1} \|_{1,1} +  \boldsymbol{\Theta}_k .
\end{align}
\textbf{Bounding $R_{2,2,k}$.}
From the update rule of $\di_k$ in~\eqref{eq:update_D} we get,
\begin{align}\label{eq:ddif}
\nonumber
\di_{k}-\di_{k-1} &= \beta_{1,k}  \Vt_k\circ \M_{k-1} +  (1-\beta_{1,k}) \Vt_k \circ \eg_k\\
\nonumber
    &-\beta_{1,k-1}  \Vt_{k-1}\circ \M_{k-2} - (1-\beta_{1,k-1}) \Vt_{k-1} \circ \eg_{k-1}\\
    \nonumber
    & =\beta_{1,k} \Vt_k\circ \M_{k-1}-\beta_{1,k-1}\Vt_{k-1}\circ \M_{k-2}\\
    \nonumber
    & + (1-\beta_{1,k}) (\Vt_k - \Vt_{k-1} + \Vt_{k-1})\circ \eg_{k}-(1-\beta_{1,k-1})\Vt_{k-1} \circ \eg_{k-1}\\
    \nonumber
    & = \beta_{1,k} \Vt_k\circ \M_{k-1} - \beta_{1,k-1} \Vt_{k-1}\circ \M_{k-2} + (1-\beta_{1,k}) (\Vt_k - \Vt_{k-1}) \circ \eg_{k}\\
    \nonumber
    &+ (1-\beta_{1,k})\Vt_{k-1}\circ (\rg_{k} +\beps_k -\bg_{k})+ (1-\beta_{1,k-1})\Vt_{k-1} \circ ( \bg_{k-1}- \rg_{k-1}-\beps_{k-1} ) \\
    &+ (1- \beta_{1,k-1}) \Vt_{k-1} \circ (\bg_{k}-\bg_{k-1}) -(\beta_{1,k}-\beta_{1,k-1}) \Vt_{k-1} \circ\bg_{k}.
\end{align}
Next, we focus on providing upper bounds for $R_{2,2,k}\leq \eta G_{\infty} \|\di_{k}-\di_{k-1}\|_F^2 $. Observe that 
\begin{subequations}
\begin{align}\label{eq:rr220}
\nonumber 
&\qquad  \left\|\beta_{1,k} \Vt_k\circ \M_{k-1}\right\|^2_F + \left\|- \beta_{1,k-1} \Vt_{k-1}\circ \M_{k-2}\right\|^2 \\
& \leq 2 \max\left(\left\|\beta_{1,k} \Vt_k\circ \M_{k-1}\right\|^2_F , \left\|- \beta_{1,k-1} \Vt_{k-1}\circ \M_{k-2}\right\|_F^2\right) \leq  \frac{2Md\beta^2_{1,k-1}}{u_c},
\end{align} 
where the inequality follows from Lemma~\ref{VM_bound}. 
Using Lemma~\ref{lem:hadamard}\ref{itm:seven}, we get
\begin{align}\label{eq:rr222} 
\left\| (1-\beta_{1,k}) (\Vt_k - \Vt_{k-1}) \circ \eg_{k}\right\|_F^2 & \leq \|\eg_k\|_{\max}^2  \|\Vt_k-\Vt_{k-1} \|_{1,1}^2 \leq G_{\infty}^2  \|\Vt_k-\Vt_{k-1} \|_{1,1}^2,
\end{align} 
where the last inequality uses Assumption~\ref{assumption:g_bounded}. Further, from Lemma~\ref{lem:hadamard}\ref{itm:eight}, we have 
\begin{align}\label{eq:rr223}
\left\|(1-\beta_{1,k})\Vt_{k-1}\circ (\rg_{k}-\bg_{k})\right\|_F^2  & \leq \|\Vt_{k-1}\|_{\max}^2 \|\rg_{k}-\bg_{k} \|^2_F  \leq \frac{L^2}{G_0^2}  \| \Z_{k} - \bar{\Z}_k \|_F^2,
\end{align} 
where the last inequality uses Assumption~\ref{assumption:function}(\ref{assu:Lipshitz_gradient}). Similarly,  
\begin{align}\label{eq:rr224}
\nonumber 
\left\|(1-\beta_{1,k-1})\Vt_{k-1}\circ (\rg_{k-1}-\bg_{k-1})\right\|^2  & \leq \|\Vt_{k-1}\|_{\max}^2 \|\rg_{k-1}-\bg_{k-1} \|^2_F  
\\
&\leq \frac{L^2}{G_0^2}  \| \Z_{k-1} - \bar{\Z}_{k-1} \|_F^2, \\ \nonumber 
\left\|(1-\beta_{1,k})\Vt_{k-1}\circ (\bg_{k}-\bg_{k-1}) \right\|_F^2  & \leq \|\Vt_{k-1}\|_{\max}^2\left\|\bg_{k}-\bg_{k-1} \right\|_F^2 \\
& \leq \frac{L^2}{G_0^2} \|\bar{\Z}_k - \bar{\Z}_{k-1}\|_F^2, \\
\nonumber 
\|(\beta_{1,k}-\beta_{1,k-1})\Vt_{k-1}\circ \bg_{k}\|_F^2  & \leq  (\beta_{1,k}-\beta_{1,k-1})^2 \|\Vt_{k-1}\|_{\max}^2 \left\|\bg_{k}\right\|^2_F \\
&\leq \frac{(\beta_{1,k}-\beta_{1,k-1})^2 }{G_0^2}\|\bg_{k}\|_F^2 \label{eq:rr226}.
\end{align} 
\end{subequations}
Taking the norm of \eqref{eq:ddif} and using \eqref{eq:rr220}--\eqref{eq:rr226}, we get 
\begin{align}\label{eq:R2khalh}
\nonumber \frac{R_{2,2,k}}{G_{\infty}}\leq \eta \left\|\di_{k}-\di_{k-1} \right\|^2 &\leq 18 \eta Md\beta^2_{1,k-1} u_c^{-1}+ 9 \eta G_{\infty}^2 \|\Vt_k-\Vt_{k-1} \|_{1,1}^2  \\
\nonumber
 & +  \frac{9 \eta L^2}{G_0^2}  \left( \| \Z_{k} - \bar{\Z}_k \|^2_F + \| \Z_{k-1} - \bar{\Z}_{k-1} \|^2_F\right) \\
 \nonumber
 & +  \frac{9 \eta}{G_0^2} \left( L^2  \|\bar{\Z}_k - \bar{\Z}_{k-1}\|^2+ (\beta_{1,k}-\beta_{1,k-1})^2 \|\bg_{k}\|^2_F\right)\\
  & +  \frac{9 \eta}{G_0^2} \left(\|\beps_{k}\|^2_F+ \|\beps_{k-1}\|^2_F\right).
\end{align}
Substituting \eqref{eq:R2khalh} and \eqref{eq:N_update} into \eqref{eq:R2k}, we obtain 
\begin{align}\label{eq:sumd0}
\nonumber
& \quad \left\|\bar\z_k-\bar\x_k\right\|^2 + \left\|\bar{\x}_{k-1}-\bar\z_k\right\|^2 \leq  \frac{R_{2,0,k}}{MG_{0}} 
\\ \nonumber
&+ \frac{2 \eta }{MG_{0}} \left(\beta_{1,k}  MD G_{\infty}\sqrt{du_c^{-1}}   + \sqrt{M}D G_{\infty}^2  \|\Vt_k-\Vt_{k-1} \|_{1,1} +  \boldsymbol{\Theta}_k \right) 
\\
\nonumber 
& + \frac{18 \eta^2 G_{\infty}}{MG_{0}}  \left(2Md\beta^2_{1,k-1}u_c^{-1} +  G_{\infty}^2  \|\Vt_k-\Vt_{k-1} \|_{1,1}^2\right) 
\\
\nonumber
 & +  \frac{18 \eta^2 L^2 G_{\infty}}{MG_0^3}  \left(  \| \Z_{k} - \bar{\Z}_k \|^2_F +  \| \Z_{k-1} - \bar{\Z}_{k-1} \|^2_F\right) 
 \\
 \nonumber
 & +  \frac{18 \eta^2 G_{\infty}}{MG_0^3} \left( L^2  \|\bar{\Z}_k - \bar{\Z}_{k-1}\|_F^2+ (\beta_{1,k}-\beta_{1,k-1})^2 \|\bg_{k}\|^2_F\right)
 \\
  & +  \frac{18 \eta^2 G_{\infty}}{MG_0^3} \left(\|\beps_{k}\|^2_F+ \|\beps_{k-1}\|^2_F\right). 
\end{align}
It follows from Lemma~\ref{lem:hadamard}\ref{itm:one} that 
\begin{align}\label{sumd1} 
\nonumber 
& \quad \sum_{k = 1}^N \|\bar\Z_k - \bar\Z_{k-1}\|_F^2 = M \sum_{k=1}^{N} \|\bar{\z}_k - \bar{\z}_{k-1}\|^2\\
\nonumber
& \leq  2M \sum_{k=1}^{N}\|\bar{\z}_k - \bar{\x}_{k-1}\|^2+  2M \sum_{k=1}^{N} \|\bar{\x}_{k-1}- \bar{\z}_{k-1}\|^2\\
& = 2M \sum_{k=1}^{N}\|\bar{\z}_k - \bar{\x}_{k-1}\|^2+  2M \sum_{k=1}^{N} \|\bar{\x}_{k}- \bar{\z}_{k}\|^2, 
\end{align}
where the equality follows since by our assumption $\bar\x_0= \bar \z_0=0$.

Now,  by using Lemmas~\ref{lem:v_diff},~\ref{lem:xstar}, summing~\eqref{eq:sumd0} over $k=1, \cdots, N$ and using \eqref{sumd1}, we obtain 

\begin{align}%\label{eq:rd2}
\nonumber  
& \quad  \left(1- \frac{36\eta^2 L^2 G_{\infty}}{G_0^3}\right)  \sum_{k=1}^{N}\E\left\|\bar{\x}_{k-1}-\bar\z_k\right\|^2 +  \left(1- \frac{36\eta^2 L^2 G_{\infty} }{G_0^3}\right) \sum_{k=1}^{N}\E\left\|\bar\z_k-\bar\x_k\right\|^2  
\\
\nonumber
& \leq   \frac{\eta \sqrt{du_c^{-1}}\rho^{2t} \beta_{1,1}}{G_0D(1-\rho^t) (1-\kappa)} +  N\eta \frac{\rho^t}{1-\rho^t} \left[G_{\infty}  D^{-1}\sqrt{d}\right] +\frac{d}{G_0\sqrt{M}}
\\
\nonumber 
&+ \frac{2 \eta D G_{\infty}}{G_0} \bigg(\frac{\beta_{1,1} }{1-\kappa} \sqrt{du_c^{-1}} +\frac{\sqrt{M} G_{\infty} d}{G_{0}^2}\bigg)    
\\
\nonumber
& + \frac{36 \eta^2 d \beta_{1,1}^2 G_{\infty}}{u_c(1-\kappa^2)G_0} +  \frac{18 \eta^2d G_{\infty}^3}{G_0^3} + \frac{36\eta^2  L^2 G_{\infty}}{MG_0^3} \sum_{k= 1}^N \sum_{i = 1}^M\E \|\z_{i,k}-\bar\z_k\|^2 
\\
& + \frac{18\eta^2 G_{\infty}}{MG_0^3}  \sum_{k=1}^{N} (\beta_{1,k} - \beta_{1,k-1})^2 \E\|\sbg_k\|^2 + \frac{36 \eta^2 G_{\infty}}{M G_{0}^3 }\sum_{k = 1}^N \frac{\sigma^2}{m_k} = :\text{R.H.S.}.
\end{align}
Note that if  $\eta \leq \sqrt{\frac{G_{0}^3}{72 L^2 G_{\infty}}}$, then  $1- \frac{36\eta^2 L^2 G_{\infty}}{G_0^3} \geq \frac{1}{2}$. Thus, from \eqref{eq:r12}, we get
\begin{align}\label{eq:rd2}
 \frac{1}{N}\sum_{k=1}^{N} \E [R_{2,k}] &=  \sum_{k=1}^{N} \E\left\|\bar{\x}_{k-1}-\bar\z_k\right\|^2 + \E\left\|\bar\z_k-\bar\x_k\right\|^2   \leq \frac{2}{N} \text{R.H.S.},  
\end{align}
\begin{align}\label{eq:revise_1}
    \frac{1}{N} \sum_{k=1}^{N} \E[R_{1,k}]  &=  \frac{3d\beta_{1,1}^2}{Nu_c(1-\kappa^2)} +  \frac{3}{NMG_0^2}\sum_{k = 1}^N \frac{\sigma^2}{m_k}+ \frac{3L^2}{NM G_0^2}\sum \limits_{k = 1}^N \sum \limits_{i = 1}^M\E\|\z_{k,i} - \bar{\z}_{k}\|^2
\end{align}

By combining \eqref{eq:revise_1} with \eqref{eq:thm_main} we obtain \begin{align*} 
&\qquad \eta^2 (1-\beta_{1,1})^2 G_{\infty}^{-2}\frac{1}{N}\sum_{k = 1}^N \E\|\sbg_k\|^2  
\\
& \leq \frac{3}{N}  \sum_{k = 1}^N  \left(\eta^2 \E[R_{1.k}] + \E[R_{2,k}] \right) \leq  \frac{ 9 \eta^2 d \beta_{1,1}^2}{Nu_c (1-\kappa^2)} + \frac{9\eta^2 }{NMG_0^2} \sum_{k = 1}^N \frac{\sigma^2}{m_k} + \frac{ 9\eta^2 L^2}{G_0^2} \bigg(\frac{1}{MN}\sum \limits_{k = 1}^N \sum \limits_{i = 1}^M\E\|\z_{k,i} - \bar{\z}_{k}\|^2 \bigg)
\\
&+  \frac{6\eta \sqrt{du_c^{-1}}\rho^{2t} \beta_{1,1}}{NG_0D(1-\rho^t) (1-\kappa)} +  6\eta \frac{\rho^t}{1-\rho^t} \left[G_{\infty}  D^{-1}\sqrt{d}\right] +\frac{6d}{NG_0\sqrt{M}}
\\
&+ \frac{12 \eta D G_{\infty}}{NG_0} \bigg(\frac{\beta_{1,1} }{1-\kappa} \sqrt{du_c^{-1}} +\frac{\sqrt{M} G_{\infty} d}{G_{0}^2}\bigg) 
\\
&+\frac{216 \eta^2 d \beta_{1,1}^2 G_{\infty}}{Nu_c(1-\kappa^2)G_0} +  \frac{108 \eta^2d G_{\infty}^3}{NG_0^3} + \frac{216\eta^2  L^2 G_{\infty}}{G_0^3}  \bigg( \frac{1}{MN}\sum \limits_{k = 1}^N \sum \limits_{i = 1}^M\E\|\z_{k,i} - \bar{\z}_{k}\|^2 \bigg)
\\
&  +\frac{108\eta^2 G_{\infty}}{MG_0^3}  \sum_{k=1}^{N} (\beta_{1,k} - \beta_{1,k-1})^2 \E\|\sbg_k\|^2 + \frac{216 \eta^2 G_{\infty}}{MN G_{0}^3}  \sum_{k = 1}^N \frac{\sigma^2}{m_k}. 
\end{align*}

 Using the fact that  $ \frac{\rho^{2t}}{1-\rho^{2t}} \leq \frac{\rho^t}{1-\rho^t}$, together with Lemma~\ref{lem:network_power_2} and rearranging the above inequality yields 
\begin{align*}
&\frac{C_0}{N}\sum_{k = 1}^N \E\|\sbg_k\|^2  \leq \frac{C_1}{N} + \frac{C_2 }{MN} \sum_{k = 1}^N \frac{\sigma^2}{m_k} + \frac{C_3 M \rho^t}{1-\rho^t},
\end{align*}
where 
\begin{align}\label{eq:revised_2}
\nonumber
& C_0 :=  \eta^2 (1-\beta_{1,1})^2 G_{\infty}^{-2} - \frac{108\beta_{1,1}^2\eta^2 G_{\infty}}{MG_0^3}
\\
\nonumber 
& C_1:=  \frac{ 9 \eta^2 d \beta_{1,1}^2}{u_c (1-\kappa^2)} + \frac{18\eta^
4 L^2 Mdu_c^{-1} \beta_{1,1}^2 \rho^{2t}}{G_0^2(1-\kappa^2)(1-\rho^{2t})}
\\ \nonumber
&  \qquad +\frac{6\eta \sqrt{du_c^{-1}}\rho^{2t} \beta_{1,1}}{D(1-\rho^t) (1-\kappa)G_0} +\frac{6d}{\sqrt{M}G_0} 
\\ \nonumber
& \qquad + \frac{12 \eta D G_{\infty}}{G_0} \bigg(\frac{\beta_{1,1} }{1-\kappa} \sqrt{du_c^{-1}} +\frac{\sqrt{M} G_{\infty} d}{G_{0}^2}\bigg) 
\\ \nonumber
& \qquad+ \frac{216 \eta^2 d \beta_{1,1}^2 G_{\infty}}{u_c(1-\kappa^2)G_0} +  \frac{108 \eta^2d G_{\infty}^3}{G_0^3} + \frac{432\eta^4  L^2 G_{\infty}}{G_0^3}  \bigg( \frac{ Mdu_c^{-1} \beta_{1,1}^2 \rho^{2t}}{(1-\kappa^2)(1-\rho^{2t})}\bigg),
\\ \nonumber
&  C_2:=   \frac{9\eta^2 }{G_0^2} + \frac{216 \eta^2 G_{\infty} }{ G_{0}^3 }, 
\\
&  C_3:=  \frac{ 18\eta^4 L^2}{G_0^2} \frac{d}{1-\beta_2}  + 6\eta \frac{G_{\infty} \sqrt{d}}{MD}  + \frac{432 \eta^4 L^2}{G_0^2 } \frac{G_{\infty}}{G_0} \frac{d}{1-\beta_2}. 
\end{align}
Then, we assume that $\beta_{1,1} \leq \frac{1}{1+\sqrt{108 G_{\infty}/{MG_0^3}}}$ and divide both sides by $ C_0$ to obtain

\begin{align*}
&\frac{1}{N}\sum_{k = 1}^N \E\|\sbg_k\|^2  \leq \frac{B_1}{N} + \frac{B_2 }{MN} \sum_{k = 1}^N \frac{\sigma^2}{m_k} + \frac{B_3 M\rho^t}{1-\rho^t},
\end{align*}
where
\begin{align}\label{eq:revised_const_global} 
B_i = \frac{C_i}{C_0},  ~~\forall i \in\{1,2,3\}.
\end{align}
This completes the proof of the Theorem.
\end{proof}
\section*{Samples generated from different algorithms}
We provide selected samples of images generated by different algorithms used in our comparisons. 

\begin{figure}[H]
\begin{subfigure}{.5\textwidth}
  \centering
  \includegraphics[width=.8\linewidth]{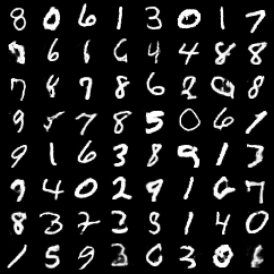}
  \caption{\DADAM}
  \label{fig:sfig1}
\end{subfigure}%
\begin{subfigure}{.5\textwidth}
  \centering
  \includegraphics[width=.8\linewidth]{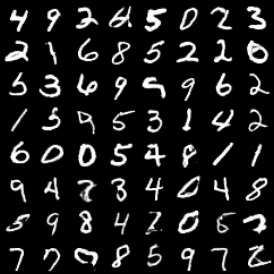}
  \caption{\CADAM}
  \label{fig:sfig1}
\end{subfigure}%

\begin{subfigure}{.5\textwidth}
  \centering
  \includegraphics[width=.8\linewidth]{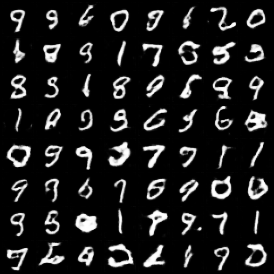}
  \caption{\DSGD}
  \label{fig:sfig1}
\end{subfigure}%
\begin{subfigure}{.5\textwidth}
  \centering
  \includegraphics[width=.8\linewidth]{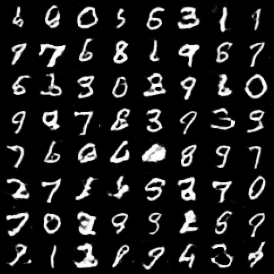}
  \caption{\CSGD}
  \label{fig:sfig1}
\end{subfigure}%

\caption{Generated MNIST Samples for Decentralized Experiment }
\label{fig:sample_mnist}

\end{figure}

\begin{figure}[H]
\begin{subfigure}{.5\textwidth}
  \centering
  \includegraphics[width=.8\linewidth]{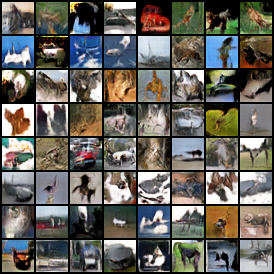}
  \caption{\DADAM}
  \label{fig:sfig1}
\end{subfigure}%
\begin{subfigure}{.5\textwidth}
  \centering
  \includegraphics[width=.8\linewidth]{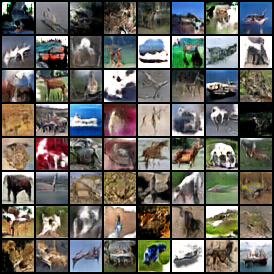}
  \caption{\CADAM}
  \label{fig:sfig1}
\end{subfigure}%

\begin{subfigure}{.5\textwidth}
  \centering
  \includegraphics[width=.8\linewidth]{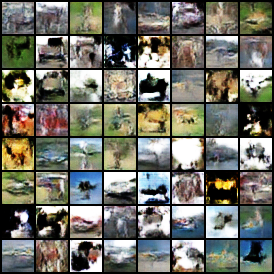}
  \caption{\DSGD}
  \label{fig:sfig1}
\end{subfigure}%
\begin{subfigure}{.5\textwidth}
  \centering
  \includegraphics[width=.8\linewidth]{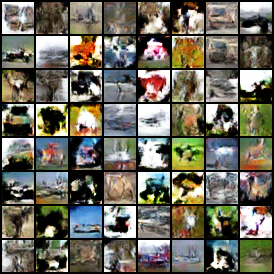}
  \caption{\CSGD}
  \label{fig:sfig1}
\end{subfigure}%

\caption{Generated CIFAR-10 Samples for Decentralized Experiment}
\label{fig:sample_cifar}
\end{figure}

\begin{figure}[H]
\begin{subfigure}{.5\textwidth}
  \centering
  \includegraphics[width=.8\linewidth]{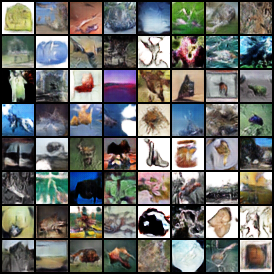}
  \caption{\DADAM}
  \label{fig:sfig100}
\end{subfigure}%
\begin{subfigure}{.5\textwidth}
  \centering
  \includegraphics[width=.8\linewidth]{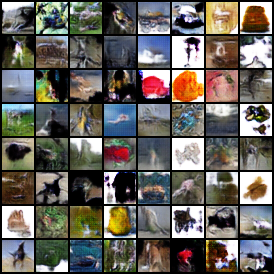}
  \caption{\CADAM}
  \label{fig:sfig100}
\end{subfigure}%

\begin{subfigure}{.5\textwidth}
  \centering
  \includegraphics[width=.8\linewidth]{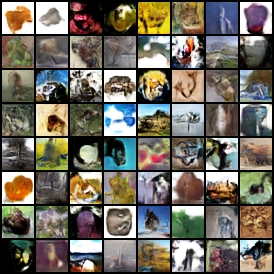}
  \caption{\DSGD}
  \label{fig:sfig100}
\end{subfigure}%
\begin{subfigure}{.5\textwidth}
  \centering
  \includegraphics[width=.8\linewidth]{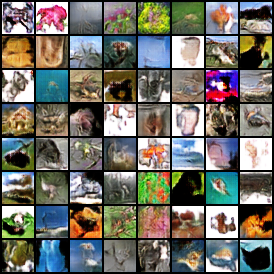}
  \caption{\CSGD}
  \label{fig:sfig100}
\end{subfigure}%

\caption{Generated CIFAR-100 Samples for Decentralized Experiment}
\label{fig:sample_cifar_100}
\end{figure}

\begin{figure}[H]
\begin{subfigure}{.5\textwidth}
  \centering
  \includegraphics[width=.8\linewidth]{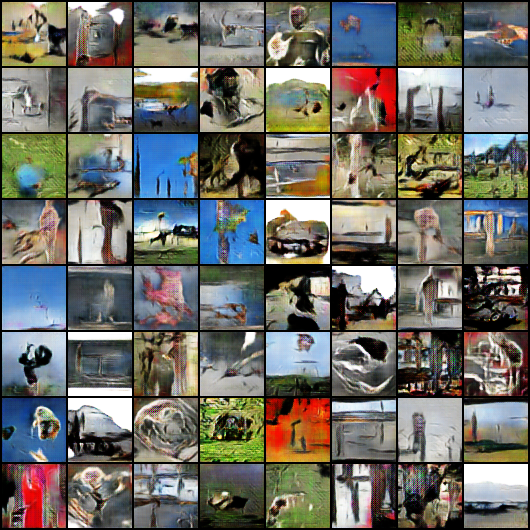}
  \caption{\DADAM}
  \label{fig:sfig100}
\end{subfigure}%
\begin{subfigure}{.5\textwidth}
  \centering
  \includegraphics[width=.8\linewidth]{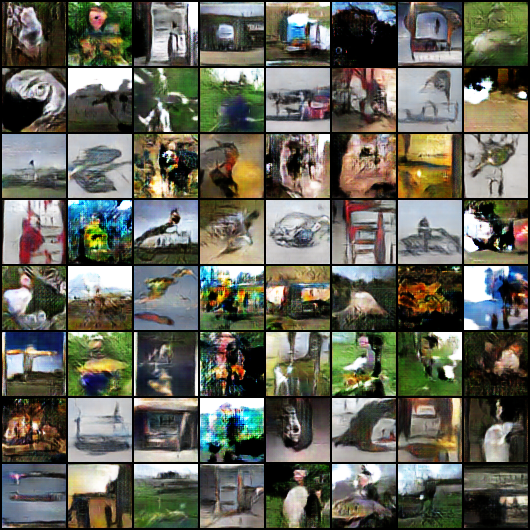}
  \caption{\CADAM}
  \label{fig:sfig100}
\end{subfigure}%

\begin{subfigure}{.5\textwidth}
  \centering
  \includegraphics[width=.8\linewidth]{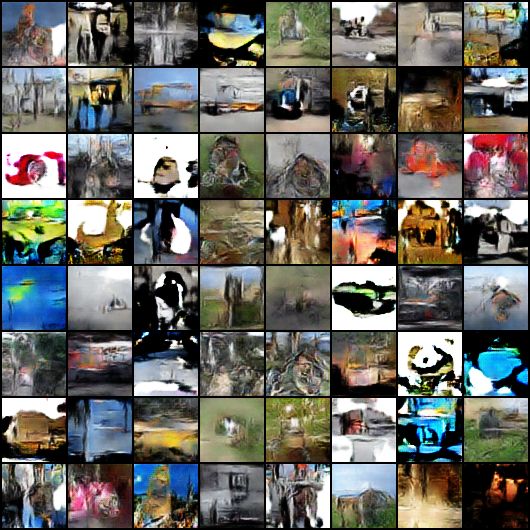}
  \caption{\DSGD}
  \label{fig:sfig100}
\end{subfigure}%
\begin{subfigure}{.5\textwidth}
  \centering
  \includegraphics[width=.8\linewidth]{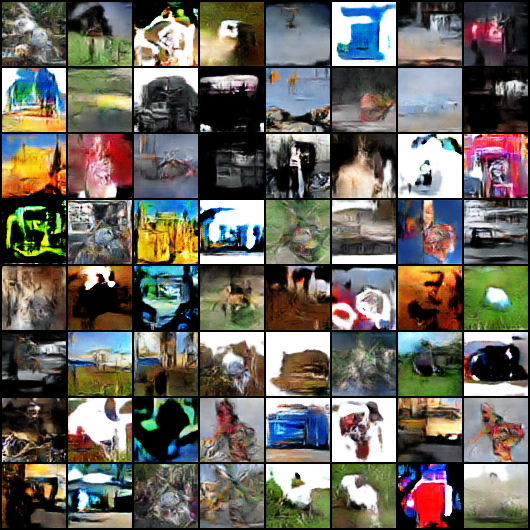}
  \caption{\CSGD}
  \label{fig:sfig100}
\end{subfigure}%

\caption{Generated Imagenette Samples for Decentralized Experiment}
\label{fig:sample_imagnet}
\end{figure}

\begin{figure}[H]
\begin{subfigure}{.5\textwidth}
  \centering
  \includegraphics[width=.8\linewidth]{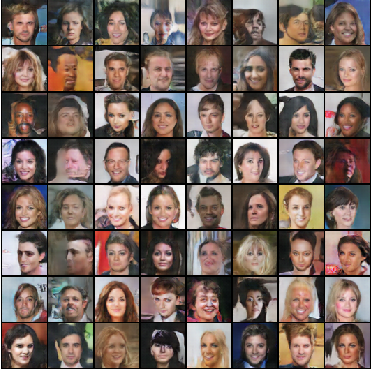}
  \caption{\ADAM}
  \label{fig:celeba_adam3}
\end{subfigure}%
\begin{subfigure}{.5\textwidth}
  \centering
  \includegraphics[width=.8\linewidth]{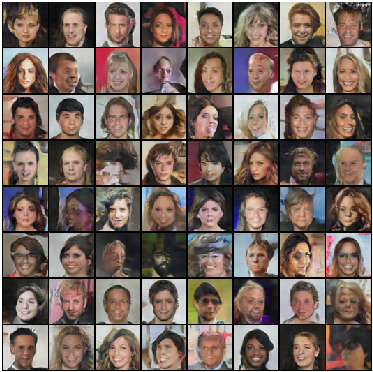}
  \caption{OAdagrad}
  \label{fig:celeba_oada}
\end{subfigure}%

\begin{subfigure}{.5\textwidth}
  \centering
  \includegraphics[width=.8\linewidth]{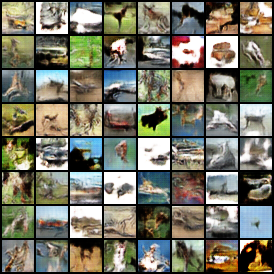}
  \caption{\ADAM}
  \label{fig:cifar10_adam3}
\end{subfigure}%
\begin{subfigure}{.5\textwidth}
  \centering
  \includegraphics[width=.8\linewidth]{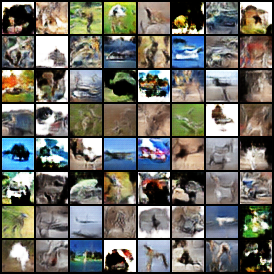}
  \caption{OAdagrad}
  \label{fig:cifar1_oada}
\end{subfigure}%

\caption{Generated CelebA (top row) and CIFAR-10 (bottom row) Samples for Single Computing Node Experiment}
\label{fig:sample_celeba_cifar}
\end{figure}

\end{document}